\documentclass[12pt,a4paper]{article}
\usepackage{latexsym,amssymb,amsfonts,amsmath,amsthm,nccmath,float,enumitem,setspace,bm,authblk}
\usepackage[usenames,dvipsnames]{xcolor}
\usepackage[hidelinks]{hyperref}
\usepackage[margin=1.75cm]{geometry}
\usepackage{bm}
\usepackage{booktabs}
\usepackage{verbatim}
\usepackage{diagbox}
\usepackage{makecell}
\usepackage{float}
\usepackage[dvipsnames]{xcolor}
\allowdisplaybreaks[1]

\hypersetup{
	colorlinks,
	linkcolor={red!80!black},
	citecolor={blue!80!black},
	urlcolor={blue!80!black}
}

\newtheorem{thm}{Theorem}[section]
\newtheorem{que}[thm]{Question}
\newtheorem{pro}[thm]{Proposition}

\newtheorem{rem}[thm]{Remark}

\newtheorem{lem}[thm]{Lemma}
\newtheorem{core}[thm]{Corollary}

\newtheorem{example}{Example}[section]%

 \def\ZZ{\mathbb Z}

\def\exp{{\sf exp}}

\let\oldproofname=\proofname
\renewcommand{\proofname}{\rm\bf{\oldproofname}}

\begin{document}
	
	\title{Perfect state transfer on Cayley graphs over dihedral groups: A complete and practical characterization}
	\author[a]{Shixin Wang}
	\affil[a]{School of Mathematical Sciences, Laboratory of Mathematics and Complex Systems, MOE, Beijing Normal University, 100875 Beijing, China}
	\affil[ ]{sxwang@bjtu.edu.cn}
	\date{}	
	\renewcommand*{\Affilfont}{\small\it}
	\maketitle
\begin{abstract}
	
Perfect state transfer on graphs has attracted extensive attention due to its application in quantum information and quantum computation.
Explicit characterizations of connection sets admitting perfect state transfer in Cayley graphs are rare and, so far, are known only for a few abelian Cayley graphs. 
In this paper, we characterize the conjugation-closed connection sets of connected Cayley graphs over dihedral groups that admit perfect state transfer. 
By applying Ramanujan sums, M\"obius inversion, and arguments based on the \(p\)-adic exponential valuation of rational numbers, we convert the eigenvalue constraints imposed by perfect state transfer into explicit structural conditions on the connection set.
This yields a complete and practical characterization, which gives an effective criterion for recognizing and constructing such Cayley graphs and also determines the exact minimum perfect state transfer time.

\end{abstract}

\section{Introduction}	

In 2000, Bennett \emph{et al.}~\cite{Bennett} pointed out that transmitting a quantum state from one location to another is a fundamental task in quantum information and quantum computation.
Since then, state transfer via quantum walks on graphs has attracted increasing attention from both physicists and mathematicians.
Bose~\cite{Bose} proposed a communication scheme using spin chains as quantum channels, thereby introducing the notion of state transfer, and Christandl \emph{et al.}~\cite{Christandl} subsequently gave the definition of perfect state transfer on graphs.
Let $\Gamma$ be a graph with adjacency matrix $A$ and vertex set $V$. 
The continuous-time quantum walk on $\Gamma$ at time $t\in\mathbb{R}_{\ge 0}$ is given by 
\[
H(t)=\exp(-\mathrm{i}tA)=\sum_{s=0}^{\infty}\frac{(-\mathrm{i}tA)^s}{s!}.
\]
The matrix $H(t)$ is called the \emph{transfer matrix} of $\Gamma$ at time $t$.
For two distinct vertices \(u,v\in V\), we say that \(\Gamma\) admits \emph{perfect state transfer} (PST for short) from \(u\) to \(v\) at time \(t\) if \(|H(t)_{u,v}|=1\).
If \(|H(t)_{u,u}|=1\), then \(\Gamma\) is said to be \emph{periodic at} \(u\)
at time \(t\).
If \(\Gamma\) is periodic at every vertex at time \(t\), then \(\Gamma\) is
\emph{periodic} at time \(t\).

Perfect state transfer provides an idealized model of lossless quantum communication and has therefore attracted considerable attention. 
Godsil \cite{State transfer on graphs} surveyed state transfer on graphs and observed that perfect state transfer is a rare phenomenon, occurring only in very few graphs.
Thus, determining whether a given graph admits perfect state transfer is one of the central problems in this area. 
To date, perfect state transfer has been studied for several important graph families, including paths \cite{Christandl,State transfer on graphs}, trees \cite{tree 1,tree 2}, cubelike graphs \cite{cubelike}, circulant graphs \cite{basic}, and graphs arising from association schemes \cite{association schemes}. 

Cayley graphs form a natural class in the study of perfect state transfer because of their strong algebraic structure.
Let \(G\) be a group and let \(S\subseteq G\setminus\{1_G\}\). 
The \emph{Cayley digraph} \(\mathrm{Cay}(G,S)\) is the digraph with vertex set \(G\) in which \((x,y)\) is an arc whenever \(yx^{-1}\in S\) and the set \(S\) is called the \emph{connection set} of \(\mathrm{Cay}(G,S)\). 
If \(S^{-1}=S\), then \(\mathrm{Cay}(G,S)\) is undirected, and is called a \emph{Cayley graph}. 
The graph \(\mathrm{Cay}(G,S)\) is connected if and only if \(\langle S\rangle=G\).
Perfect state transfer on Cayley graphs has been studied for various classes of groups, including abelian groups \cite{Tan}, dihedral groups \cite{dihedral,Perfect state transfer on Cayley graphs over dihedral group}, semidihedral groups \cite{semi-dihedral}, dicyclic groups \cite{dicyclic}, a class of Cayley graphs obtained as certain extensions of complete graphs \cite{NEPS}. 
However, these works mainly focus on the constraints imposed on the eigenvalues by perfect state transfer, rather than on explicit descriptions of the connection sets for which perfect state transfer occurs.

To the best of our knowledge, explicit characterizations of connection sets of Cayley graphs admitting perfect state transfer are known only in two classes. 
Ba\v{s}i\'c \cite[Theorem 22]{basic} obtained such a characterization for circulant graphs, namely Cayley graphs over cyclic groups, by using number-theoretic methods, while \'Arnad\'ottir and Godsil \cite[Theorem 9.1]{On state transfer in Cayley graphs for abelian groups} did so for Cayley graphs over abelian groups with cyclic Sylow-$2$ subgroup by using group-theoretic methods.

In this paper, we establish the first complete connection-set characterization of perfect state transfer for Cayley graphs over a family of non-abelian groups.
More precisely, we completely determine the conjugation-closed connection sets for which a connected Cayley graph over a dihedral group admits perfect state transfer.
Our approach combines character-theoretic eigenvalue formulas for Cayley graphs, Ramanujan sums, M\"obius inversion on the divisor lattice, and the \(p\)-adic exponential valuation of rational numbers. 
The key idea is to derive explicit eigenvalue expressions and then work backwards from the eigenvalue constraints imposed by perfect state transfer to obtain structural restrictions on the connection sets. 
In this way, we extend the connection-set characterizations previously obtained by Ba\v{s}i\'c and by \'Arnad\'ottir and Godsil to the non-abelian case of Cayley graphs over dihedral groups. 
Combining Theorems \ref{thm:8midn-characterization}, \ref{thm:n4mod8-full-reflection} and \ref{thm:n4mod8-half-reflection} with Corollaries~\ref{cor:minimal-time-8midn}, \ref{cor:minimal-time-n4mod8} and~\ref{core: minimal time 2}, and including the description of the perfect state transfer pair, we obtain the following summary of our main results.

\begin{thm}\label{thm:1.1}
	Let \(G\cong D_{2n}\), $G=\langle a,b: a^n=b^2=1,b^{-1}ab=a^{-1}\rangle$ and \(S=S_0\cup S_1\subseteq G\setminus\{1_G\}\) be conjugation-closed, where $S_0=S\cap \langle a\rangle$, $S_1=S\cap b\langle a\rangle$.
	Let $\Gamma=\mathrm{Cay}(G,S)$ be connected.
	When \(4\mid n\), set $P:=\{d:d\mid n,\ v_2(\frac nd)=2\}
	\setminus\left\{\frac n4\right\}$.
	Then \(\Gamma\) admits perfect state transfer if and only if $n$ is even and one of the following holds.
		\begin{enumerate}
	\item[\rm(1)] \(8\mid n\), and there exist $E\subseteq P$, $F\subseteq \{d:d\mid n,\ v_2(\frac nd)\ge 3\}$ such that
	\[
	S_0=\bigcup_{d\in D}G_n(d),
	\quad
	D=F\cup E\cup 2E\cup 4E\cup\{\xi\},\quad \xi\in\left\{\frac n4,\frac n2\right\}.
	\]
	In this case, the only restriction on \(S_1\) is that it is nonempty.
	\item[\rm(2)] \(n\equiv 4\pmod 8\), there exists \(E\subseteq P\) such that either $S_1=b\langle a\rangle$ and
	\[
	S_0=\bigcup_{d\in D}G_n(d),
	\quad
	D=E\cup 2E\cup 4E\cup\{\xi\},\quad \xi\in\left\{\frac n4,\frac n2\right\}
	\]
or $S_1\in\{b\langle a^2\rangle,ba\langle a^2\rangle\}$ and there exists \(\rho\in\{0,1\}\) with \(E\ne\emptyset\) or \(\rho=1\) such that
	\[
S_0=\bigcup_{d\in D}G_n(d),
\quad
D=E\cup 2(P\setminus E)\cup 4(P\setminus E)
\cup \rho\left\{\frac n4,\frac n2\right\}.
\]		
	\item[\rm (3)] \(n\equiv 2 \pmod 4\), $S=G\setminus\{1_G,a^\frac{n}{2}\}$.
	\end{enumerate}
	Moreover, if \(\Gamma\) admits perfect state transfer between distinct vertices \(g\) and \(g'\),
	then \(g'=ga^{\frac n2}\), and the minimum time at which perfect state transfer occurs is $\frac{\pi}{2}$.
\end{thm}

Theorem \ref{thm:1.1} significantly strengthens previous results such as \cite[Theorem~3.2]{Perfect state transfer on Cayley graphs over dihedral group}.
In that result, one must first compute all eigenvalues and then verify several \(2\)-adic conditions, while the minimum perfect state transfer time is given only implicitly through a graph-dependent parameter (see Lemma~\ref{lem:iff of PST}).  
By contrast, Theorem \ref{thm:1.1} reduces the problem to an immediate structural check (see Remarks \ref{rem:1} and \ref{rem:2}).
This makes the criterion more direct and easier to apply in practice. 
It also shows that the minimum perfect state transfer time is \(\frac{\pi}{2}\), and hence is graph-independent.

The paper is organized as follows. Section \ref{sec:Preliminaries} collects the preliminaries used throughout. In Section \ref{sec:Eigenvalues of Cayley Graphs on dihedral groups in terms of Ramanujan sum}, we derive Ramanujan-sum expressions for the eigenvalues of integral Cayley graphs over dihedral groups,
and use M\"obius inversion on the divisor lattice to characterize a class of integral circulant graphs. 
Section \ref{sec:Perfect state transfer of Cayley Graphs on dihedral groups} gives a reformulation of the perfect state transfer criterion. 
Section \ref{sec:Characterization} establishes complete characterizations of the connection sets admitting perfect state transfer, according to whether \(n\equiv 0 \pmod{4}\) or \(n\equiv 2 \pmod{4}\). Section \ref{sec:examples} illustrates the applicability of these results through several examples. Section \ref{sec:7} concludes with remarks and an open problem.

\section{Preliminaries}	\label{sec:Preliminaries}

In this section, we present the preliminary results used throughout the paper. 
All graphs are assumed to be simple and undirected, and all Cayley graphs over dihedral groups considered are assumed to have conjugation-closed connection sets. 
The main object of this paper is connected Cayley graphs over dihedral groups.
The auxiliary integral circulant graphs appearing in the proofs are not assumed to be connected unless this is explicitly stated.

\subsection{$p$-adic exponential valuation of rational numbers}

For a prime $p$, the {\em $p$-adic exponential valuation of rational numbers} is a mapping defined by
$$v_p:\mathbb{Q}\longrightarrow \ZZ\cup \{\infty\},~v_p(0)=\infty,~v_p(p^l\frac{a}{b})=l,~{\text where} ~a,b,l\in \ZZ~{\text and}~p\nmid ab.$$
We assume that $\infty + \infty= \infty + l=\infty$ and $\infty\ge  l$ for any $l\in \ZZ$.
Then $v_p$ has the following properties.
For $\beta,\beta'\in \mathbb{Q}$,
\begin{itemize}
	\item [(P1)] $v_p(\beta \beta')=v_p(\beta)+v_p(\beta')$;
	\item [(P2)] $v_p(\beta + \beta')\ge  \min(v_p(\beta),v_p(\beta')), \mbox{ and the equality
		holds if }v_p(\beta)\neq v_p(\beta').$
\end{itemize}

Throughout this paper, only the case \(p=2\) will be needed, and the above two properties will be used repeatedly, often without explicit mention.

\subsection{Ramanujan sums}
For a positive integer \(n\) and a non-negative integer \(j\), the  {\it Ramanujan sum} $c(j,n)$ is defined as
\begin{equation}\label{equation:definition}
c(j,n) = \sum_{\substack{k=1\\ \gcd(k,n)=1}}^{n} \exp(2\pi \mathrm{i}\frac{kj}{n}).
\end{equation}
Moreover, $c(j,n) $ has the expression
\begin{equation}\label{equation:Ramanujan sum}
c(j,n) = \mu(t_{n,j}) \frac{\varphi(n)}{\varphi(t_{n,j})}, \quad t_{n,j} = \frac{n}{\gcd(n,j)},
\end{equation}
where $\mu$ is the M\"obius function defined as
\[
\mu(n) =
\begin{cases}
	1, & \text{if } n = 1, \\
	0, & \text{if } n \text{ is not square-free}, \\
	(-1)^k, & \text{if } n \text{ is the product of } k \text{ distinct prime numbers}.
\end{cases}
\]

For later use, we record a simple reduction property of Ramanujan sums.
\begin{lem}\label{lem:ramanujan-odd}
	Let $h$, $m$, $r$, $s$ and $j$ be positive integers, where $h$ and $m$ are odd, $\gcd(r,s)=1$. Then 
	\begin{enumerate}
		\item[$(1)$] $c(h,2m)=-c(h,m)$ and $c(h,4m)=0$.	
		\item[$(2)$] $c(2,2m)=\mu(m)$ and $c(2,4m)=-2\mu(m)$, moreover, $\mu(2m)=-\mu(m)$.
		\item[$(3)$] $c(2j,m)=c(j,m)$.	
		\item[$(4)$] $c(j,rs)=c(j,r)c(j,s)$.
	\end{enumerate}
\end{lem}
\begin{proof}
	For (1), since $m$ and $h$ are odd, we have $\gcd(2m,h)=\gcd(m,h)$,	and hence $t_{2m,h}=\frac{2m}{\gcd(2m,h)}=2\cdot \frac{m}{\gcd(m,h)}=2t_{m,h}$.
	Since $m$ is odd, the integer $t_{m,h}$ is odd. 
	By \eqref{equation:Ramanujan sum},	we obtain
	\[
	c(h,2m)
	=\mu(2t_{m,h})\frac{\varphi(2m)}{\varphi(2t_{m,h})}=-\mu(t_{m,h})\frac{\varphi(m)}{\varphi(t_{m,h})}=-c(h,m).
	\]
	Similarly, $t_{4m,h}=\frac{4m}{\gcd(4m,h)}=4t_{m,h}$, 	which is divisible by $4$, so $\mu(t_{4m,h})=0$. 
	Therefore $c(h,4m)=0$.
	
	For (2), since $m$ odd, $\gcd(2,2m)=2$ and $\gcd(2,4m)=2$.
	Thus $t_{2m,2}=m$ and $t_{4m,2}=2m$.
	Hence
	\[
	c(2,2m)=\mu(m)\frac{\varphi(2m)}{\varphi(m)}=\mu(m),\quad
	c(2,4m)=\mu(2m)\frac{\varphi(4m)}{\varphi(2m)}=2\mu(2m)=-2\mu(m).
	\]
	
	For (3), since $m$ is odd, we have $\gcd(m,2j)=\gcd(m,j)$.
	Hence $t_{m,2j}=\frac{m}{\gcd(m,2j)}	=\frac{m}{\gcd(m,j)}=t_{m,j}$.
	Therefore, 
	\[
	c(2j,m)
	=\mu(t_{m,2j})\frac{\varphi(m)}{\varphi(t_{m,2j})}
	=\mu(t_{m,j})\frac{\varphi(m)}{\varphi(t_{m,j})}
	=c(j,m).
	\]
	
	For (4), let $t_{r,j}=\frac{r}{\gcd(r,j)}$ and  $t_{s,j}=\frac{s}{\gcd(s,j)}$.
	Since $\gcd(r,s)=1$, we have $\gcd(rs,j)=\gcd(r,j)\gcd(s,j)$, and thus $t_{rs,j}=\frac{rs}{\gcd(rs,j)}=\frac{r}{\gcd(r,j)}\cdot \frac{s}{\gcd(s,j)}=t_{r,j}t_{s,j}$.
	Moreover, $\gcd(t_{r,j},t_{s,j})=1$. Since both $\mu$ and $\varphi$ are multiplicative on coprime arguments, we obtain
	\[
	c(j,rs)
	=\mu(t_{rs,j})\frac{\varphi(rs)}{\varphi(t_{rs,j})}
	=\mu(t_{r,j})\mu(t_{s,j})
	\frac{\varphi(r)\varphi(s)}{\varphi(t_{r,j})\varphi(t_{s,j})}
	=c(j,r)c(j,s).
	\]
	This completes the proof.
\end{proof}

\begin{lem}\label{lem:basic lem 17}
	Let \(m\) and \(j\) be positive integers with \(j\) even. Then
	\[
	c(j,m)=
	\begin{cases}
		c \left(\frac{j}{2},m\right), & \text{if}\  v_2(m)=0,\\[5pt]
		c \left(\frac{j}{2},\frac{m}{2}\right), & \text{if}\ v_2(m)=1,\\[5pt]
		2\,c \left(\frac{j}{2},\frac{m}{2}\right), & \text{if}\ v_2(m)\ge 2.
	\end{cases}
	\]
\end{lem}

\begin{proof}
Write \(j=2t\) and \(m=2^s q\), where \(q\) is odd and \(s=v_2(m)\).
By Lemma \ref{lem:ramanujan-odd}-(4), we have $c(2t,m)=c(2t,2^s)c(2t,q)$.
Since \(q\) is odd, Lemma \ref{lem:ramanujan-odd}-(3) yields $c(2t,q)=c(t,q)$, and so $c(2t,m)=c(2t,2^s)c(t,q)$. 
	
If \(s=0\), then \(m=q\) and $v_2(m)=0$, we have $c(2t,m)=c(t,m)$.
	
If \(s=1\), then \(m=2q\) and $v_2(m)=1$, and since \(c(2t,2)=1\), we get $c(2t,m)=c(t,q)=c(t,\frac{m}{2})$.
		
If \(s\ge 2\), then $v_2(m)\ge 2$ and $t_{2^s,\,2t}
=
\frac{2^s}{\gcd(2^s,2t)}
=
\frac{2^{s-1}}{\gcd(2^{s-1},t)}
=
t_{2^{s-1},\,t}$.
Hence, by \eqref{equation:Ramanujan sum}, we obtain $c(2t,2^s)
=
\mu(t_{2^s,2t})\frac{\varphi(2^s)}{\varphi(t_{2^s,2t})}
=
\mu(t_{2^{s-1},t})\frac{2\varphi(2^{s-1})}{\varphi(t_{2^{s-1},t})}
=
2\,c(t,2^{s-1})$.
Hence $c(2t,m)=c(2t,2^s)c(t,q)=2c(t,2^{s-1})c(t,q)= 2c(t,\frac{m}{2})$.
This completes the proof.
\end{proof}

\begin{lem}\label{lem: c(j,d)}
Let \(j\) and \(d\) be positive integers. 
Then $c(j,d)=\sum\limits_{r\mid \gcd(j,d)} r\,\mu \left(\frac{d}{r}\right)$.
\end{lem}
\begin{proof}
Using the identity (see \cite[Theorem 8.16]{Rosen})
$
\sum\limits_{r\mid \gcd(k,d)}\mu(r)=
\begin{cases}
	1,& \gcd(k,d)=1\\
	0,& \gcd(k,d)>1
\end{cases}
$ and \eqref{equation:definition}, we obtain
\[
c(j,d)=\sum_{\substack{1\le k\le d\\ \gcd(k,d)=1}} \exp(2\pi \mathrm{i} \frac{kj}{d})
=\sum_{k=1}^{d}\left(\sum_{r\mid \gcd(k,d)}\mu(r)\right)\exp(2\pi \mathrm{i} \frac{kj}{d})
=\sum_{r\mid d}\mu(r)\sum_{\substack{1\le k\le d\\ r\mid k}} \exp(2\pi \mathrm{i} \frac{kj}{d}).
\]
Write \(k=rt\). 
Then
\[
\sum_{\substack{1\le k\le d,~r\mid k}} \exp(2\pi \mathrm{i} \frac{kj}{d})
=\sum_{t=1}^{\frac dr} \exp(2\pi \mathrm{i} \frac{rtj}{d}).
\]
This is a sum of all \(\frac{d}{r}\)-th roots of unity. Hence it equals \(\frac{d}{r}\) if \(\frac{d}{r}\mid j\), and equals \(0\) otherwise.
Therefore $c(j,d)=\sum\limits_{\substack{r\mid d,~\frac{d}{r}\mid j}}\mu(r)\frac{d}{r}$.
Now let \(s=\frac{d}{r}\). 
Then \(s\mid d\) and \(s\mid j\), that is, \(s\mid \gcd(j,d)\). 
Therefore $c(j,d)=\sum\limits_{s\mid \gcd(j,d)} s\,\mu \left(\frac{d}{s}\right)$.
This completes the proof.\qedhere

\end{proof}

\subsection{Representations and characters of dihedral groups}

We briefly recall some standard notions from the representation theory of finite groups.
The reader is referred to \cite{Steinberg} for more details on representations of groups.

Let \(G\) be a finite group. 
A representation of \(G\) is a homomorphism $\rho:G\to GL(V)$, where \(V\) is a finite-dimensional complex vector space. 
The dimension of \(V\) is called the degree of \(\rho\). The character of \(\rho\) is the function $\chi_\rho:G\to \mathbb{C}$, $\chi_\rho(g)=\operatorname{Tr}(\rho(g))$, where ``Tr'' denotes the trace of a linear transformation. 
A representation is called irreducible if the only \(G\)-invariant subspaces of \(V\) are \(\{0\}\)
and \(V\).
We recall only the irreducible representations and characters of \(D_{2n}\) needed later, in the case where \(n\) is even.

\begin{lem}\label{lem:representations and character table of D2n}
Let \(n\) be even. 
The irreducible representations of \(D_{2n}=\langle a,b: a^n=b^2=1,b^{-1}ab=a^{-1}\rangle\) and their corresponding characters are listed in Tables \ref{tab:representations_Dn_even} and \ref{tab:char_table_D2n_even}, where \(\omega = \exp \left(\frac{2\pi i}{n}\right)\) and \(\chi_i\) denotes the character afforded by \(\rho^{(i)}\).
	\begin{table}[htbp]
		\centering
		\begin{tabular}{c|c|c}
			\hline
			& \(a^k\ (0 \le  k \le  n-1)\) & \(ba^k\ (0 \le  k \le  n-1)\) \\
			\hline
			\(\rho^{(1)}\) & \(1\) & \(1\) \\	
			\(\rho^{(2)}\) & \(1\) & \(-1\)\\ 	
			\(\rho^{(3)}\) & \((-1)^k\) & \((-1)^k\) \\
			\(\rho^{(4)}\) & \((-1)^k\) & \((-1)^{k+1}\) \\ 	
			\(\rho^{(l)}\  (5 \le  l \le  3+\frac{n}{2})\) & \(\begin{pmatrix}\omega^{(l-4)k} & 0 \\ 0 & \omega^{-(l-4)k}\end{pmatrix}\) & \(\begin{pmatrix}0 & \omega^{-(l-4)k} \\ \omega^{(l-4)k} & 0\end{pmatrix}\) \\
			\hline
		\end{tabular}
		\vspace{0.5em} 
		\caption{Representations of \(D_{2n}\) when \(n\) is even.}
		\label{tab:representations_Dn_even}
	\end{table}
	
	\begin{table}[htbp]
		\centering
		\begin{tabular}{c|c|c}
			\hline
			& \(a^k\ (0 \le  k \le  n-1)\) & \(ba^k\ (0 \le  k \le  n-1)\) \\
			\hline
			\(\chi_1\) & \(1\) & \(1\) \\ 	
			\(\chi_2\) & \(1\) & \(-1\) \\ 	
			\(\chi_3\) & \((-1)^k\) & \((-1)^k\) \\
			\(\chi_4\) & \((-1)^k\) & \((-1)^{k+1}\) \\ 	
			\(\chi_l\ (5 \le  l \le  3+\frac{n}{2})\) & \(2\cos\left(\frac{2(l-4)k\pi}{n}\right)\) & \(0\) \\
			\hline 
		\end{tabular}
		\vspace{0.5em} 
		\caption{The character table of \(D_{2n}\) when \(n\) is even.}
		\label{tab:char_table_D2n_even}
		\vspace{0.2em}
	\end{table}
\end{lem}

\section{Eigenvalues of Cayley graphs in terms of Ramanujan sums}\label{sec:Eigenvalues of Cayley Graphs on dihedral groups in terms of Ramanujan sum}

In this section, we first determine the eigenvalues of Cayley graphs over dihedral groups with conjugation-closed connection sets. 
When such a graph is integral, these eigenvalues can be expressed in terms of Ramanujan sums.
Separately, we use Ramanujan sums and M\"obius inversion on the divisor lattice to derive a structural property of odd-order integral circulant graphs that will be needed later.

We begin by describing the form of an inverse-closed and conjugation-closed connection set \(S\) in \(D_{2n}=\langle a,b: a^n=b^2=1,b^{-1}ab=a^{-1}\rangle\), where $n$ is even.
The conjugacy classes of \(D_{2n}\) are $\{1_{D_{2n}}\}$ together with reflection conjugacy classes 
\[
\mathcal R_0=\{ba^{2j}:0\le  j\le  \frac{n}{2}-1\},\quad
\mathcal R_1=\{ba^{2j+1}:0\le  j\le  \frac{n}{2}-1\},
\]
and rotation conjugacy classes
\[
\{a^{\frac{n}{2}}\},\quad
\{a^k,a^{-k}\},\quad 1\le  k\le  \frac{n}{2}-1.
\]
Since \(1_{D_{2n}}\notin S\), any inverse-closed and conjugation-closed connection set $S$ can be expressed as
\begin{equation}\label{equ:S}
S=\left(\bigcup_{k\in K}\{a^k,a^{-k}\}\right)\cup\bigl(\epsilon\{a^{\frac{n}{2}}\}\bigr)\cup(\alpha \mathcal R_0)\cup(\beta \mathcal R_1),	
\end{equation}
where $K\subseteq\{1,\dots,\frac{n}{2}-1\}$,
$\epsilon,\alpha,\beta\in\{0,1\}$.
For convenience, we always write $S=S_0\cup S_1$, where
\begin{equation}\label{equ:S0S1}
 S_0 = S \cap \langle a \rangle=
\left(
\bigcup_{k\in K}\{a^k,a^{-k}\}
\right)
\cup
\bigl(\epsilon\{a^{\frac{n}{2}}\}\bigr),\quad S_1 = S \cap b\langle a \rangle=(\alpha \mathcal R_0)
\cup
(\beta \mathcal R_1)
\end{equation}

Recall that \(\mathrm{Cay}(G,S)\) is connected if and only if \(\langle S\rangle=G\).
Thus, in the notation above, the case \(\alpha=\beta=0\) is disconnected. 
If \(\alpha+\beta=1\), then connectedness is equivalent to \(S_0\) containing an odd power of \(a\), or equivalently, to \(D\) containing an odd divisor when \(S_0=\bigcup\limits_{d\in D}G_n(d)\).

Next we recall the character-theoretic eigenvalue formula for Cayley graphs with conjugation-closed connection sets and an integrality criterion for such graphs.

\begin{lem}\cite{Zieschang}\label{lem:eigenvalues of quasi-abelian graph}
Let $n\ge  3$ be even and $G=D_{2n}$ whose character table is shown in Tables $\ref{tab:representations_Dn_even}$ and $\ref{tab:char_table_D2n_even}$.
Let $S\subseteq G$ be conjugation-closed. 	
Then the eigenvalues of Cayley graph \( \operatorname{Cay}(G, S) \) are given by
\begin{equation*}\label{eq:eigenvalue}
	\lambda_k = \frac{1}{d_k} \sum_{g \in S} \chi_k(g), \quad 1 \le  k \le  3+\frac{n}{2},
\end{equation*}
where $d_k$ denotes the degree of the irreducible representation affording $\chi_k$ and each \( \lambda_k \) has multiplicity \( d_k^2 \).
\end{lem}

\begin{lem}\cite[Corollary 4.1]{Perfect state transfer on Cayley graphs over dihedral group}\label{lem:D2n integral-0}
Let \(n\) be a positive integer and let \(S\subseteq D_{2n}\) be conjugation-closed.
Then \( \Gamma=\mathrm{Cay}(D_{2n},S) \) is integral if and only if \( S_0^\ell = S_0 \) for all \( \ell \in \mathbb{Z}_n^* \).
\end{lem}

\begin{core}\label{core:eigenvalues of D2n}
	With the notation in Lemma \ref{lem:D2n integral-0}, \( \Gamma\) is integral if and only if there exists a set of positive divisors \(D \subseteq \{ d: d \mid n, d<n \}\) such that $S_0 = \bigcup\limits_{d \in D} G_n(d)$,
	where $G_n(d) =  \{ a^k : \gcd(k, n) =d,\ 1\le  k<  n\}$.
	Moreover, if $\Gamma$ is integral, then its eigenvalues are given by
	\begin{equation}
		\begin{split}
			\lambda_1 &=A+\frac n2(\alpha+\beta),~~\lambda_2 = A-\frac n2(\alpha+\beta),\\
			\lambda_3 &=T+\frac n2(\alpha-\beta),~~\lambda_4 =T+\frac n2(-\alpha+\beta),\\
			\eta_h &= \sum_{d \in D} c(h, \frac nd),\quad 1 \le  h \le  \frac{n}{2}-1,
		\end{split}
	\end{equation}
	where $A=|S_0|$, $T=\sum\limits_{a^k\in S_0}(-1)^k$, and \(c(\cdot,\cdot)\) denotes the Ramanujan sum.
\end{core}
\begin{proof}
	Suppose first that $S_0=\bigcup\limits_{d\in D}G_n(d)$, where $G_n(d)=\{a^k:\gcd(k,n)=d,\ 1\le  k<  n\}$.
	Let $\ell\in\mathbb Z_n^*$. For any $a^k\in G_n(d)$, since $\gcd(\ell,n)=1$, we have $\gcd(\ell k,n)=\gcd(k,n)=d$.
	Hence $(a^k)^\ell=a^{\ell k}\in G_n(d)$, and so $G_n(d)^\ell\subseteq G_n(d)$. Applying the same argument to $\ell^{-1}$ yields $G_n(d)\subseteq G_n(d)^\ell$. Thus $G_n(d)^\ell=G_n(d)$ for every $d\mid n$, and therefore
	\[
	S_0^\ell=\bigcup_{d\in D}G_n(d)^\ell=\bigcup_{d\in D}G_n(d)=S_0.
	\]
	It follows from Lemma \ref{lem:D2n integral-0} that \( \Gamma\) is integral.
	
	Conversely, assume that $S_0^\ell=S_0$ for all $\ell\in\mathbb Z_n^*$. 
	Let
	\[
	D=\{d:\text{there exists }a^k\in S_0\text{ with }\gcd(k,n)=d\}.
	\]
	Clearly, $S_0\subseteq\bigcup\limits_{d\in D}G_n(d)$. To prove the reverse inclusion, take $a^k\in S_0$ and $a^m\in G$ with $\gcd(m,n)=\gcd(k,n)=d$. Write $k=dk_1$, $m=dm_1$, $n=dn_1$, where $\gcd(k_1,n_1)=\gcd(m_1,n_1)=1$.
	Then there exists $u\in\mathbb Z_{n_1}^*$ such that $uk_1\equiv m_1 \pmod{n_1}$.
	Choose $\ell\in\mathbb Z_n^*$ satisfying $\ell\equiv u\pmod{n_1}$. 
	Then $\ell k\equiv m \pmod n$, and hence $(a^k)^\ell=a^m$. 
	Since $a^k\in S_0$ and $S_0^\ell=S_0$, we obtain $a^m\in S_0$. 
	Therefore $G_n(d)\subseteq S_0$ for every $d\in D$, and so $S_0=\bigcup\limits_{d\in D}G_n(d)$.
	Hence we complete the proof of first statement.
	
	For second statement, since $S$ is conjugation-closed, $S$ is as described in \eqref{equ:S}.
	Let $A=|S_0|$, $T=\sum\limits_{a^k\in S_0}(-1)^k$.
    Then by Lemma \ref{lem:eigenvalues of quasi-abelian graph}, Tables \ref{tab:representations_Dn_even} and \ref{tab:char_table_D2n_even}, we have
   \begin{equation*}
   	\begin{split}
   	\lambda_1 &= \sum\limits_{g \in S} 1 = A+\frac n2(\alpha+\beta),\quad
   	\lambda_2 = \sum\limits_{a^k \in S_0} 1 + \sum\limits_{b a^k \in S_1} (-1)= A-\frac n2(\alpha+\beta),\\
   	\lambda_3 &= \sum\limits_{a^k \in S_0} (-1)^k + \sum\limits_{b a^k \in S_1} (-1)^k= T + \frac{n}{2}(\alpha - \beta),\\
   	\lambda_4 &= \sum\limits_{a^k \in S_0} (-1)^{k} + \sum\limits_{b a^k \in S_1} (-1)^{k+1}= T + \frac{n}{2}(-\alpha + \beta),
   	\end{split}
   \end{equation*}
    and for \(5 \le  l \le  3+\frac{n}{2}\), set $h=l-4$,
    \begin{equation*}
    	\lambda_l = \frac{1}{2} \sum_{g \in S} \chi_{\rho^{(l)}}(g) = \frac{1}{2} \sum_{a^k \in S_0} (\omega^{hk} + \omega^{-hk}).
    \end{equation*}
If $\Gamma$ is integral, then by first statement, there exists a set $D\subseteq \{d:d\mid n, d<n\}$ such that $S_0=\bigcup\limits_{d\in D}G_n(d)$,
where $G_n(d)=\{a^k:\gcd(k,n)=d,\ 1\le  k<  n\}$.
Hence
\[
\lambda_l
=\frac12\sum_{a^k\in S_0}(\omega^{hk}+\omega^{-hk})
=\frac12\sum_{d\in D}\sum_{a^k\in G_n(d)}(\omega^{hk}+\omega^{-hk}).
\]
For a fixed $d\in D$, if $a^k\in G_n(d)$, then $a^{-k}\in G_n(d)$.
Therefore
$$\sum_{a^k\in G_n(d)}(\omega^{hk}+\omega^{-hk})=2\sum_{a^k\in G_n(d)}\omega^{hk}=2c(h,\frac nd),$$
and so $\lambda_l
=\sum\limits_{d\in D}c(h,\frac nd)$.
Setting $\eta_h:=\lambda_{h+4}$ for $1\le  h\le  \frac n2-1$, the proof is complete.
\end{proof}

For later use, we adopt the standard notation for \emph{integral circulant graphs}.
Let $n$ be a positive integer and $E\subseteq \{d:d\mid n,\ 1\le d<n\}$.
Define
\[
\mathrm{ICG}_n(E):=\mathrm{Cay} \left(\mathbb Z_n,\bigcup_{d\in E} G_n(d)\right),
\]
where $G_n(d)=\{a^k:\gcd(k,n)=d,\ 1\le k\le n-1\}$.
It is well known that the eigenvalues of \(\mathrm{ICG}_n(E)\) are
\[
\vartheta_j=\sum_{d\in E} c \left(j,\frac nd\right),\quad 0\le j\le n-1.
\]
In particular, if \(X=\mathrm{Cay}(\langle a\rangle,S_0)\) and $S_0=\bigcup\limits_{d\in D}G_n(d)$, then \(X=\mathrm{ICG}_n(D)\).
Thus, throughout the rest of the paper, whenever we consider integral circulant graphs, their eigenvalues will always be understood through the above Ramanujan sum expression.

We conclude this section with a property of integral circulant graphs that will be very useful in the analysis of the case \(n\equiv 2 \pmod 4\).

\begin{lem}\label{lem:auxiliary-Km}
	Let \(m\) be odd, let \(E\subseteq \{e:e\mid m,\ e<m\}\), and let $X=\mathrm{ICG}_m(E)$.
	Let $\tau_j=\sum\limits_{e\in E}c(j,\frac me)$, $0\le j\le m-1$, be the eigenvalues of $X$.
	If every \(\tau_j\) with \(1\le j\le m-1\) is odd, then $E=\{\,e:e\mid m,\ e<m\,\}$.
	Equivalently, $X=K_m$.
\end{lem}
\begin{proof}
	If \(m=1\), then \(E=\emptyset\) and \(X=K_1\), so the conclusion is trivial. Hence we may assume \(m>1\).
	Let $D:=\left\{\frac{m}{e}:e\in E\right\}\subseteq \{d:d\mid m,\ d>1\}$.
	Then $\tau_j=\sum\limits_{e\in E}c \left(j,\frac{m}{e}\right)=\sum\limits_{d\in D} c(j,d)$, $0\le j\le m-1$. 
	By Lemma \ref{lem: c(j,d)},	we obtain
	\[
	\tau_j=\sum_{d\in D} c(j,d)
	=\sum_{d\in D}\sum_{r\mid \gcd(j,d)} r\,\mu \left(\frac{d}{r}\right)
	=\sum_{r\mid \gcd(j,m)} r \sum_{\substack{d\in D,~r\mid d}}
	\mu \left(\frac{d}{r}\right).
	\]
	For each divisor \(r\) of \(m\), write $\alpha_r:=\sum\limits_{\substack{d\in D,~r\mid d}}
	\mu \left(\frac{d}{r}\right)$.
	Then $\tau_j=\sum\limits_{r\mid \gcd(j,m)} r\alpha_r$.
	Since \(m\) is odd, all its divisors are odd, so \(r\equiv 1 \pmod 2\). 
	Therefore
	\begin{equation}\label{equa: tau j}
		\tau_j=\sum_{r\mid \gcd(j,m)} r\alpha_r
		\equiv \sum_{r\mid \gcd(j,m)} \alpha_r \pmod 2.
	\end{equation}
	
	Taking $j=1$ in \eqref{equa: tau j}, we have $\tau_1\equiv \alpha_1\pmod 2$.
	Since $\tau_1$ is odd by assumption, $\alpha_1$ is odd. 
	We next show that $\alpha_r$ is even for every divisor \( r\) of \(m\) with \(1<r<m\), by induction on \( r\), ordered increasingly.
	Suppose $p\neq 1$ is the minimal proper divisor of $m$.
	Taking $j=p$ in \eqref{equa: tau j}, we have  $\tau_p\equiv\sum\limits_{r\mid p}\alpha_r\pmod 2\equiv \alpha_1+\alpha_p\pmod 2$.
	Since \(\alpha_1\) is odd and $\tau_p$ is odd, we have $\alpha_p$ is even.
	Now let \(t>1\) be a proper divisor of \(m\), and assume inductively that $\alpha_d$ is even for every proper divisor \(d\) of \(m\) with \(1<d<t\). 
	Taking \(j=t\) in \eqref{equa: tau j}, we obtain
	\[\tau_t\equiv \sum_{r\mid t}\alpha_r\pmod 2\equiv \alpha_1+\alpha_t + \sum_{\substack{r \mid t,~1<r < t}} \alpha_r\pmod 2.\]
	By the induction hypothesis, $\alpha_r$ is even for $r\mid t$ and $1<r<t$.
	Since \(\alpha_1\equiv 1\pmod 2\) and $ \tau_t$ is odd, it follows that $1\equiv \tau_t\equiv 1+\alpha_t \pmod 2$.
	Hence $\alpha_r$ is even for every divisor \( r\) of \(m\) with \(1<r<m\).
	Moreover, the integer \(\frac me>1\) is odd. 
	Hence each \(\varphi(\frac me)\) is even, and so \(\tau_0\) is even.
	Now we consider $\alpha_m$.
	Taking \(j=0\) in \eqref{equa: tau j} gives $0\equiv \tau_0\equiv \sum\limits_{r\mid m}\alpha_r \pmod 2$.
	By $\alpha_1$ is odd and $\alpha_r$ is even for $r\mid t$ and $1<r<t$, we obtain $0\equiv 1+\alpha_m \pmod 2$.
	Therefore $\alpha_m\equiv 1\pmod 2$.
	
Now, for each divisor \(r\) of \(m\), let \(\mathbf{1}_D(r)\) denote the indicator function of \(D\), namely, \(\mathbf{1}_D(r)=1\) if \(r\in D\), and \(\mathbf{1}_D(r)=0\) otherwise. 
	Then, for each divisor \(r\) of \(m\),
	\[
	\alpha_r
	=\sum_{\substack{d\in D,~r\mid d}}\mu \left(\frac{d}{r}\right)
	=\sum_{\substack{t\mid m,~r\mid t}}\mathbf{1}_{D}(t)\,\mu \left(\frac{t}{r}\right).
	\]
	Next we use M\"obius inversion on the divisor lattice of \(m\), that is, if $g(r)=\sum\limits_{\substack{t\mid m,~r\mid t}} f(t)\,\mu \left(\frac{t}{r}\right)$ for every $r\mid m$, then $f(r)=\sum\limits_{\substack{t\mid m,~r\mid t}} g(t)$.
	Applying this with \(f=\mathbf{1}_D\) and \(g=\alpha\), we obtain
	\[
	\mathbf{1}_{D}(r)=\sum_{\substack{t\mid m,~r\mid t}}\alpha_t.
	\]
	Now every divisor \(t\) of \(m\) satisfying \(r\mid t\) also satisfies \(t\ge r>1\). Hence every such \(t\neq m\) is a proper divisor of \(m\) greater than \(1\), so \(\alpha_t\) is even. 
	Therefore $\mathbf{1}_{D}(r)\equiv \alpha_m\equiv 1\pmod 2$.
	Since \(\mathbf{1}_{D}(r)\in\{0,1\}\), it follows that \(\mathbf{1}_{D}(r)=1\) for every divisor \(r\) of \(m\) with \(r>1\).
	Thus $D=\{d:d\mid m,\ d>1\}$.
	Equivalently, $E=\left\{\frac{m}{d}:d\mid m,\ d>1\right\}
	=\{e:e\mid m,\ e<m\}$.
	Hence $X=\mathrm{ICG}_m(E)=K_m$.
	This completes the proof.
\end{proof}

\section{Revisiting PST on Cayley graphs over dihedral groups}\label{sec:Perfect state transfer of Cayley Graphs on dihedral groups}

We now recall the eigenvalues criterion for perfect state transfer on Cayley graphs over dihedral
groups. 
The following lemma is a reformulation of \cite[Theorem~3.2]{Perfect state transfer on Cayley graphs over dihedral group}, see also \cite[Example~6.1]{Perfect state transfer on quasi-abelian semi-Cayley graphs}.
Throughout this section, we keep the notation introduced in Corollary \ref{core:eigenvalues of D2n}.

\begin{lem}\label{lem:iff of PST}\cite{Perfect state transfer on Cayley graphs over dihedral group,Perfect state transfer on quasi-abelian semi-Cayley graphs}
	Let $n\ge  3$ and $G\cong D_{2n}$, $G=\left\langle a,b\mid a^n=b^2=1,b^{-1}ab=a^{-1}\right\rangle$, $S\subseteq G$ be conjugation-closed and $\Gamma=\mathrm{Cay}(G,S)$.
	Then $\Gamma$ has perfect state transfer between distinct vertices $g$ and $g'$ at time $t$ if and only if $n$ is even and $\Gamma$ is integral and there exists an integer $\gamma$ such that one of the following holds.
	\begin{itemize}
		\item[$(1)$] If $n\equiv 2 \pmod 4$, then
		\begin{align*}
			v_2(\lambda_1-\lambda_3)&=v_2(\lambda_1-\lambda_4)=\gamma,\quad
			v_2(\lambda_1-\lambda_2)>\gamma,\\
			v_2(\eta_1-\lambda_1)&=\gamma,\quad
			v_2(\eta_{j+1}-\eta_j)=\gamma,\quad 1\le  j\le  \frac n2-2.	
		\end{align*}		
		\item[$(2)$] If $n\equiv 0 \pmod 4$, then
		\begin{align*}
			&v_2(\lambda_1-\lambda_2)>\gamma,\quad
			v_2(\lambda_1-\lambda_3)>\gamma,\quad
			v_2(\lambda_1-\lambda_4)>\gamma,\\
			&v_2(\eta_1-\lambda_1)=\gamma,\quad
			v_2(\eta_{j+1}-\eta_j)=\gamma,\quad 1\le  j\le  \frac n2-2.
		\end{align*}
	\end{itemize}
Moreover, if $\Gamma$ has perfect state transfer between distinct vertices $g$ and $g'$, then $g'g^{-1}=a^{\frac{n}{2}}$, and the minimum time at which perfect state transfer occurs is $\frac{\pi}{M}$ where $M=\gcd(\lambda_1-\lambda:\lambda\neq \lambda_1)$.
\end{lem}

\begin{core}\label{core:parity}
	If $\Gamma$ admits perfect state transfer, then $\eta_1,\eta_2,\dots,\eta_{\frac n2 -1}$ are either all of the same parity or
	alternating in parity.
\end{core}
\begin{proof}
	By Lemma \ref{lem:iff of PST}, there exists an integer \(\gamma\) such that $v_2(\eta_{j+1}-\eta_j)=\gamma$, $1\le j\le \frac n2-2$.
	If \(\gamma=0\), then each difference \(\eta_{j+1}-\eta_j\) is odd, so \(\eta_j\) and \(\eta_{j+1}\) have opposite parity for every \(j\). 
	Hence the sequence \(\eta_1,\eta_2,\ldots,\eta_{\frac n2-1}\) alternates in parity. 
	If \(\gamma\ge 1\), then each difference \(\eta_{j+1}-\eta_j\) is even, and therefore \(\eta_j\equiv \eta_{j+1}\pmod 2\) for every \(j\).
	Hence all terms \(\eta_j\) have the same parity.
\end{proof}

\section{Characterization of the connection sets}\label{sec:Characterization}

Throughout this section, we retain the notation introduced above, in particular that of Corollary~\ref{core:eigenvalues of D2n}. 

Suppose that \(\Gamma\) admits perfect state transfer. Then, by Lemma~\ref{lem:iff of PST}, \(\Gamma\) is integral. Hence, by Corollary~\ref{core:eigenvalues of D2n}, there exists a divisor set $D\subseteq \{d:d\mid n,\ d<n\}$ such that
\[
S_0=\bigcup_{d\in D}G_n(d).
\]
Since the conjugation-closed connection sets \(S\) have the explicit form given in \eqref{equ:S0S1}, the problem reduces to determining the admissible divisor sets \(D\) and the possible values of \(\alpha,\beta\). 
The main idea is to combine the Ramanujan-sum expressions for the eigenvalues in Corollary~\ref{core:eigenvalues of D2n} with the \(2\)-adic conditions in Lemma~\ref{lem:iff of PST}. 
This converts the eigenvalues conditions for perfect state transfer into explicit structural conditions on \(D,\alpha,\beta\).

For \(0\le i\le v_2(n)\), define
\[
D_i:=\{d\in D: v_2\left(\frac nd \right)=i\},
\]
and write
\[
D_{\le i}:=\bigcup_{j\le i}D_j,\quad D_{\ge i}:=\bigcup_{j\ge i}D_j.
\]
With this notation fixed, we treat separately the cases \(n\equiv0\pmod4\) and \(n\equiv2\pmod4\).

We first record a parity criterion for the odd-position eigenvalues of an integral circulant graph.
Although \cite[Lemmas~12 and~13]{basic} are stated for connected integral circulant graphs, their proofs do not use connectedness. 
Hence the same conclusions hold for arbitrary integral circulant graphs \(\mathrm{ICG}_n(E)\).

\begin{lem}\cite[Lemmas~12 and~13]{basic}\label{lem:basic lem12+13}
	Let \(E\subseteq\{d:d\mid n,\ 1\le d<n\}\), and let
	\(Z=\mathrm{ICG}_n(E)\). For \(i\ge 0\), define $E_i=\{d\in E:v_2(\frac nd)=i\}$, $E_{\le 1}:=E_0\cup E_1$, $E_1^*:=E_1\setminus\{\frac n2\}$.
	Let $\vartheta_j=\sum\limits_{d\in E}c(j,\frac nd)$, $0\le j\le n-1$, be the eigenvalues of \(Z\). 
	Then the following hold.
	\begin{enumerate}
		\item[$(1)$] All eigenvalues \(\vartheta_j\) with \(j\) odd are even if and only if $E_{\le 1}=E_1\cup 2E_1$.		
	    \item[$(2)$] All eigenvalues \(\vartheta_j\) with \(j\) odd are odd if and only if $E_{\le 1}=E_1^*\cup 2E_1^*\cup\left\{\frac n2\right\}$.
	\end{enumerate}
\end{lem}

\subsection{The case \(n\equiv 0\pmod 4\)}
If \(d\in D_{\ge 2}\), then \(v_2( \frac nd)\ge 2\).
By Ramanujan sum expression, we have \(c(h, \frac nd)=0\) for odd \(h\), and hence \(\eta_h\), for odd $h$, depends only on \(D_{\le1}\), and may therefore be analyzed through \(\mathrm{ICG}_n(D_{\le1})\).	

\begin{lem}\label{thm:-1 0}
	Let \(n\equiv 0\pmod 4\), \(G\cong D_{2n}\), and let
	\(\Gamma=\mathrm{Cay}(G,S)\). Suppose that \(\Gamma\) admits perfect state transfer.
	Then, for every odd integer \(h\) with \(1\le h\le \frac n2-1\), \(\eta_h=-1\) if \(\frac n2\in D_1\), and \(\eta_h=0\) if \(\frac n2\notin D_1\).
	Consequently, $D_0=2(D_1\setminus\{\frac n2\})$.
\end{lem}

\begin{proof}
	Let $X:=\mathrm{ICG}_n(D_{\le 1})$, and denote its eigenvalues by $\theta_j=\sum\limits_{d\in D_{\le 1}}c(j,\frac nd)$, $0\le j\le n-1$.
	
	We first compare \(\eta_h\) with the odd-position eigenvalues of \(X\).
	Let \(h\) be odd. 
	If \(d\in D_{\ge 2}\), then the Ramanujan-sum formula gives $c(h,\frac nd)=0$.
	Thus, for every odd \(h\) with \(1\le h\le \frac n2-1\), $\eta_h =\sum\limits_{d\in D}c(h,\frac nd)=\sum\limits_{d\in D_{\le 1}}c(h,\frac nd)=\theta_h$.
	Since \(\frac n2\) is even, every odd-position eigenvalue of \(X\) is one of the numbers \(\theta_h\) or \(\theta_{n-h}\), where \(h\) is odd and \(1\le h\le \frac n2-1\). 
	Since \(X\) is undirected,  $\theta_{n-h}=\theta_h=\eta_h$.
	Therefore the odd-position eigenvalues of \(X\) are precisely represented by \(\eta_h\) with \(h\) odd.
	
	By Corollary \ref{core:parity}, all \(\eta_h\) with odd \(h\) have the same parity. We now
	determine this parity according to whether \(\frac n2\) belongs to \(D_1\).
	
	If \(\frac n2\in D_1\), we claim that all \(\eta_h\) with \(h\) odd are odd. 
	Otherwise they are all even, and so all odd-position eigenvalues of \(X\) are even.
	By Lemma \ref{lem:basic lem12+13}-(1), we obtain $D_{\le 1}=D_1\cup 2D_1$.
	But \(\frac n2\in D_1\) implies $n=2\cdot \frac n2\in 2D_1\subseteq D_{\le 1}\subseteq D$, which is impossible, since \(D\subseteq\{d:d\mid n,\ d<n\}\). 
	Hence all \(\eta_h\) with \(h\) odd are odd.
	Since all odd-position eigenvalues of \(X\) are odd, Lemma \ref{lem:basic lem12+13}-(2) gives $D_{\le 1}=D_1^*\cup 2D_1^*\cup\{\frac n2\}$, $D_1^*:=D_1\setminus\{\frac n2\}$.
	Therefore, for odd \(h\),
	\[
	\eta_h=\sum_{d\in D_1^*}\left(c\left(h,\frac nd\right)+c\left(h,\frac n{2d}\right)\right) + c(h,2).
	\]
	For \(d\in D_1^*\), the integer \(\frac n{2d}\) is odd. 
	Since \(h\) is odd,
	Lemma \ref{lem:ramanujan-odd}-(1) gives $c(h,\frac nd)=-c(h,\frac n{2d})$.
	Thus all paired terms cancel.
	Moreover, \(c(h,2)=-1\). 
	Hence $\eta_h=-1$ for every odd \(h\).
	
	If \(\frac n2\notin D_1\), we claim that all \(\eta_h\) with \(h\) odd are even. 
	Otherwise they are all odd. 
	Put $Y:=\mathrm{ICG}_n(D_{\le 1}\cup\{\frac n2\})$,	and denote its eigenvalues by \(\widetilde{\theta}_j\). 
	For odd \(h\) with \(1\le h\le \frac n2-1\), we have $\widetilde{\theta}_h=\eta_h+c(h,2)=\eta_h-1$.
	Thus all these \(\widetilde{\theta}_h\)'s are even. Since \(\frac n2\) is even and \(Y\) is undirected, all odd-position eigenvalues of \(Y\) are even. 
	By Lemma \ref{lem:basic lem12+13}-(1), $D_{\le 1}\cup\{\frac n2\}=(D_1\cup\{\frac n2\})\cup	2(D_1\cup\{\frac n2\})$.
	The right-hand side contains \(n=2\cdot \frac n2\), whereas the left-hand side contains only proper divisors of \(n\). 
	This contradiction shows that all \(\eta_h\) with
	\(h\) odd are even.
	Then all odd-position eigenvalues of \(X\) are even. 
	By Lemma \ref{lem:basic lem12+13}-(1), $D_{\le 1}=D_1\cup 2D_1$.
	Hence, for odd \(h\),
	\[
	\eta_h
	=
	\sum_{d\in D_1}
	\left(
	c\left(h,\frac nd\right)
	+
	c\left(h,\frac n{2d}\right)
	\right).
	\]
	Again \(\frac n{2d}\) is odd and $c(h,\frac nd)=-c(h,\frac n{2d})$.
	Thus all paired terms cancel, and so $\eta_h=0$ for every odd \(h\).
	
	Finally, in the first case we have $D_{\le 1}=D_1^*\cup 2D_1^*\cup\{\frac n2\}$, $D_1^*=D_1\setminus\{\frac n2\}$, and in the second case we have $D_{\le 1}=D_1\cup 2D_1$.
	Since \(D_{\le 1}=D_0\cup D_1\), both cases give $D_0=2(D_1\setminus\{\frac n2\})$.
	This completes the proof.
\end{proof}

The values of the odd-indexed \(\eta_h\)'s obtained in Lemma \ref{lem:iff of PST} allow us to rewrite the
\(2\)-adic conditions in Lemma \ref{lem:iff of PST} in terms of the even-indexed eigenvalues.

\begin{lem}\label{lem:basic lem16}
	Let $n \equiv 0 \pmod{4}$, $G \cong D_{2n}$, and $\Gamma=\mathrm{Cay}(G,S)$.
	Suppose $\Gamma$ admits perfect state transfer.
	Then, in the notation of Lemma \ref{lem:iff of PST}, for $1\le  j\le  \frac n4-1$, we have $v_2(\eta_{2j})=\gamma$ if $\frac n2\notin D_1$,
	and $v_2(\eta_{2j}+1)=\gamma$ if $\frac n2\in D_1$.
\end{lem}
\begin{proof}
	Since \(\Gamma\) admits perfect state transfer, Lemma  \ref{lem:iff of PST} yields $v_2(\eta_{j+1}-\eta_j)=\gamma$, $1\le j\le \frac n2-2$.
	If \(\frac n2\notin D_1\), then Lemma \ref{thm:-1 0} implies \(\eta_{2j-1}=0\), and hence $v_2(\eta_{2j})=v_2(\eta_{2j}-\eta_{2j-1})=\gamma$
	for every \(1\le j\le \frac n4-1\).
	If \(\frac n2\in D_1\), then Lemma \ref{thm:-1 0} implies \(\eta_{2j-1}=-1\), and hence $v_2(\eta_{2j}+1)=v_2(\eta_{2j}-\eta_{2j-1})=\gamma$ for every \(1\le j\le \frac n4-1\). 
	This completes the proof.
\end{proof}

We now turn to the even-indexed eigenvalues \(\eta_{2j}\). 
By applying Lemma \ref{lem:basic lem 17}, we rewrite them in terms of the eigenvalues of a suitable integral circulant graph of order \(\frac n2\). 
This gives the following relation between \(D_1\) and \(D_2\).

\begin{pro}\label{pro:D1-D2-8midn}
	Let \(8\mid n\), \(G\cong D_{2n}\), and let
	\(\Gamma=\mathrm{Cay}(G,S)\). If \(\Gamma\) admits perfect state transfer, then $D_1\setminus\{\frac n2\}=2(D_2\setminus\{\frac n4\})$.
\end{pro}

\begin{proof}
	Put \(n_1=\frac n2\). 
	Set $E=(D_1\setminus\{\frac n2\})\cup D_2$, and let $\sigma_j=\sum\limits_{d\in E}c(j,\frac{n_1}{d})$, $0\le j\le n_1-1$,	be the eigenvalues of \(\mathrm{ICG}_{n_1}(E)\). 
	Also put $\varepsilon=1$ if $\frac n2\in D_1$, and $\varepsilon=0$ if $\frac n2\notin D_1$.
	
	Let \(j\) be odd and \(1\le j\le \frac n4-1\). 
	We compute \(\eta_{2j}\) from the decomposition \(D=D_0\cup D_1\cup D_2\cup D_{\ge 3}\).
	By Lemma \ref{thm:-1 0}, $D_0=2(D_1\setminus\{\frac n2\})$.
	Thus for each \(e\in D_1\setminus\{\frac n2\}\), the two divisors \(e\in D_1\) and \(2e\in D_0\) occur together. 
	By Lemma \ref{lem:basic lem 17}, their total contribution to \(\eta_{2j}\) is $c(2j,\frac ne)+c(2j,\frac n{2e})=2c(j,\frac {n_1}e)$.
	Similarly, each \(d\in D_2\) contributes $c(2j,\frac nd)=	2c(j,\frac {n_1}d)$.
	The divisors in \(D_{\ge 3}\) contribute nothing, since \(j\) is odd and	\(v_2(\frac n{2d})\ge 2\) for \(d\in D_{\ge 3}\). 
	Finally, the possible divisor \(\frac n2\in D_1\) contributes \(c(2j,2)=1\). 
	Therefore
	\begin{align*}
		\eta_{2j}=2\sum_{e\in D_1\setminus\{\frac n2\}}
		c\left(j,\frac {n_1}e\right)
		+2\sum_{d\in D_2}c\left(j,\frac {n_1}d\right)
		+\varepsilon =2\sigma_j+\varepsilon .
	\end{align*}

	If \(\frac n2\notin D_1\), then \(\varepsilon=0\), and Lemma \ref{lem:basic lem16} gives $v_2(2\sigma_j)=v_2(\eta_{2j})=\gamma$.
	If \(\frac n2\in D_1\), then \(\varepsilon=1\), and Lemma \ref{lem:basic lem16} gives $v_2(2(\sigma_j+1))=v_2(\eta_{2j}+1)	=\gamma$.
	In both cases, the numbers \(\sigma_j\), with \(j\) odd and
	\(1\le j\le \frac n4-1\), have the same parity.
	
	Since \(8\mid n\), the integer \(\frac n4=\frac{n_1}{2}\) is even.
	Hence every	odd-position eigenvalue of \(\mathrm{ICG}_{n_1}(E)\) is one of the numbers \(\sigma_j\) or \(\sigma_{n_1-j}\), where \(j\) is odd and \(1\le j\le \frac n4-1\). 
	Since \(\mathrm{ICG}_{n_1}(E)\) is undirected, $\sigma_{n_1-j}=\sigma_j$.
	Thus all odd-position eigenvalues of \(\mathrm{ICG}_{n_1}(E)\) have the same parity.
	
	For the graph \(\mathrm{ICG}_{n_1}(E)\), a divisor \(d\in E\) satisfies $v_2(\frac{n_1}{d})=0$ if and only if $d\in D_1\setminus\{\frac n2\}$, and it satisfies $v_2(\frac{n_1}{d})=1$ if and only if $d\in D_2$.
	If all odd-position eigenvalues of \(\mathrm{ICG}_{n_1}(E)\) are even, then Lemma \ref{lem:basic lem12+13} gives $E=D_2\cup 2D_2$. 
	Since \(E\) contains only proper divisors of \(n_1\), we must have
	\(\frac n4\notin D_2\). 
	Comparing the elements \(d\) satisfying
	\(v_2(\frac{n_1}{d})=0\) on both sides, we obtain $D_1\setminus\{\frac n2\}	=2D_2=2(D_2\setminus\{\frac n4\})$.
	Now suppose that all odd-position eigenvalues of \(\mathrm{ICG}_{n_1}(E)\) are odd. 
	We first show that \(\frac n4\in D_2\). Suppose not. 
	Set $\widetilde E=E\cup\{\frac n4\}$.
	For every odd \(j\), the added divisor \(\frac n4\) contributes $c(j,2)=-1$.
	Thus all odd-position eigenvalues of \(\mathrm{ICG}_{n_1}(\widetilde E)\) are even. 
	By Lemma \ref{lem:basic lem12+13}, $\widetilde E=(D_2\cup\{\frac n4\})
	\cup2(D_2\cup\{\frac n4\})$.
	But the right-hand side contains $2\cdot \frac n4=n_1$, which is not a proper divisor of \(n_1\). This contradiction shows that \(\frac n4\in D_2\).
	
	Now Lemma \ref{lem:basic lem12+13} gives $E=(D_2\setminus\{\frac n4\})\cup2(D_2\setminus\{\frac n4\})\cup\{\frac n4\}$.
	Comparing the elements \(d\) satisfying \(v_2(\frac{n_1}{d})=0\) on both sides, we obtain $D_1\setminus\{\frac n2\}=2(D_2\setminus\{\frac n4\})$.
	This completes the proof.
\end{proof}

It remains to determine which of the two special divisors \(\frac n4\) and \(\frac n2\) belongs to \(D\). 
The next proposition settles this point.

\begin{pro}\label{pro:structure-8midn}
	Let \(8\mid n\), \(G\cong D_{2n}\), and suppose that
	\(\Gamma=\mathrm{Cay}(G,S)\) admits perfect state transfer. 
	Put $D_2^*=D_2\setminus\{\frac n4\}$.
	Then $D_1\setminus \{\frac n2 \}=2D_2^*$, $D_0=4D_2^*$, and exactly one of \(\frac n4\) and \(\frac n2\) occurs in \(D\).
	Consequently, $D=D_{\ge 3}\cup D_2^*\cup 2D_2^*\cup 4D_2^*\cup\{\xi\}$, where $\xi\in\{\frac n4,\frac n2\}$.
    Moreover, \(\gamma=1\), $$
    \eta_{2j-1}=
    \begin{cases}
    	0, & \xi=\frac n4\\
    	-1, & \xi=\frac n2
    \end{cases}
    \ (1\le j\le \frac n4),\ \text{and}\ \eta_{2j}\equiv
    \begin{cases}
    	2, & \xi=\frac n4\\
    	1, & \xi=\frac n2
    \end{cases}
    \pmod 4,
    \ (1\le j\le \frac n4-1)$$
\end{pro}

\begin{proof}
	By Lemma \ref{thm:-1 0}, $D_0=2(D_1\setminus\{\frac n2\})$.
	By Proposition~\ref{pro:D1-D2-8midn}, $D_1\setminus\{\frac n2\}=2(D_2\setminus\{\frac n4\})=2D_2^*$.
	Thus $D_0=4D_2^*$.
	Hence $D=D_{\ge 3}\cup D_2^*\cup 2D_2^*\cup 4D_2^*\cup	\varepsilon\{\frac n2\}\cup\delta\{\frac n4\}$,
	where
	\[
	\varepsilon=
	\begin{cases}
		1, & \text{if } \frac n2\in D_1,\\
		0, & \text{if } \frac n2\notin D_1,
	\end{cases}
	\quad
	\delta=
	\begin{cases}
		1, & \text{if } \frac n4\in D_2,\\
		0, & \text{if } \frac n4\notin D_2.
	\end{cases}
	\]
	
	We now determine \(\varepsilon\) and \(\delta\). 
	By Lemma \ref{thm:-1 0}, $\eta_1=-\varepsilon$.
	We compute \(\eta_2\). 
	For each \(d\in D_2^*\), put $m_d=\frac n{4d}$.
	Then \(m_d\) is odd, and the three divisors \(d,2d,4d\) occur in $D$ together.
	By Lemma \ref{lem:ramanujan-odd},
	\begin{align*}
		c\left(2,\frac nd\right)+c\left(2,\frac n{2d}\right)+	c\left(2,\frac n{4d}\right)=-2\mu(m_d)+\mu(m_d)+\mu(m_d)=0.
	\end{align*}
	The divisors in \(D_{\ge 3}\) also make no contribution to \(\eta_2\). The
	possible divisor \(\frac n2\in D_1\) contributes $c(2,2)=1$, and the possible divisor \(\frac n4\in D_2\) contributes $c(2,4)=-2$.
	Therefore $\eta_2=\varepsilon-2\delta$.
	It follows that $\eta_2-\eta_1=2(\varepsilon-\delta)$.
	By Lemma \ref{lem:iff of PST}, \(v_2(\eta_2-\eta_1)=\gamma\), where \(\gamma\) is an integer.
	Hence \(\eta_2-\eta_1\neq 0\), and so \(\varepsilon\neq\delta\). Since
	\(\varepsilon,\delta\in\{0,1\}\), exactly one of them is equal to \(1\). 
	Thus exactly one of \(\frac n4\) and \(\frac n2\) occurs in \(D\), and $\gamma=v_2(2(\varepsilon-\delta))=1$.
	
	If \(\xi=\frac n4\), then \((\varepsilon,\delta)=(0,1)\). 
	By Lemmas \ref{thm:-1 0} and  \ref{lem:basic lem16}, $\eta_{2j-1}=0$, where $1\le j\le \frac n4$, and $v_2(\eta_{2j})=1$, where $1\le j\le \frac n4-1$.
	Hence $\eta_{2j}\equiv 2 \pmod 4$, $1\le j\le \frac n4-1$.
	
	If \(\xi=\frac n2\), then \((\varepsilon,\delta)=(1,0)\). 
	By Lemmas \ref{thm:-1 0} and \ref{lem:basic lem16}, $\eta_{2j-1}=-1$, where $1\le j\le \frac n4$, and $v_2(\eta_{2j}+1)=1$, where $1\le j\le \frac n4-1$.
	Hence $\eta_{2j}\equiv 1 \pmod 4$, $1\le j\le \frac n4-1$.
	This completes the proof.
\end{proof}

The preceding results determine the required structure of \(D_0,D_1\) and \(D_2\). 
Next lemma shows that \(D_{\ge3}\) does not affect the transfer matrix at time \(\frac \pi 2\).

\begin{lem}\label{lem:D3 PST}
	Assume that $S_0=\bigcup\limits_{d\in D}G_n(d)$.
	Let $S_{\ge  3}:=\bigcup\limits_{d\in D_{\ge  3}}G_n(d)$, $\Gamma_{\ge  3}:=\mathrm{Cay}(D_{2n},S_{\ge  3})$,
	and $A_{\ge  3}$ be the adjacency matrix of $\Gamma_{\ge  3}$.
	Then every eigenvalue of $\Gamma_{\ge  3}$ is divisible by $4$.
	In particular, $\exp(-\frac{\pi\mathrm{i}}2A_{\ge  3})=I$.
\end{lem}
\begin{proof}
	Since \(S_{\ge 3}\subseteq \langle a\rangle\), we have \(S_1=\emptyset\) for \(\Gamma_{\ge 3}\). 
	Hence by
	Corollary \ref{core:eigenvalues of D2n}, its eigenvalues are
	\[
	|S_{\ge  3}|,\ |S_{\ge  3}|,\ \sum_{a^k\in S_{\ge  3}}(-1)^k,\ \sum_{a^k\in S_{\ge  3}}(-1)^{k},\quad 
	\Xi_h=\sum_{d\in D_{\ge  3}}c \left(h,\frac nd\right),\quad 1\le  h\le  \frac n2-1.
	\]
	We will show that, the above eigenvalues are all divisible by $4$.
	
	For $d\in D_{\ge  3}$, we have $v_2(\frac nd)\ge  3$, so $8\mid \frac nd$.
	Thus $\varphi \left(\frac nd\right)\in 4\mathbb Z$, and so $|S_{\ge  3}|=\sum\limits_{d\in D_{\ge  3}}\varphi \left(\frac nd\right)\in 4\mathbb Z$.
	
	Next fix \(d\in D_{\ge 3}\). 
	If \(a^k\in G_n(d)\), then
	\(\gcd(k,n)=d\), and hence \(k\) and \(d\) have the same parity. 
	Therefore \((-1)^k\) is constant on \(G_n(d)\), and
	\[
	\sum_{a^k\in G_n(d)}(-1)^k
	=
	\begin{cases}
		\varphi\left(\frac nd\right), & d\text{ even},\\
		-\varphi\left(\frac nd\right), & d\text{ odd}.
	\end{cases}
	\]
	Since \(d\in D_{\ge 3}\), we have \(\varphi(\frac nd)\in 4\mathbb Z\). 
	Hence every \(G_n(d)\) contributes a multiple of \(4\) to
	\(\sum\limits_{a^k\in S_{\ge3}}(-1)^k\), and therefore $\sum\limits_{a^k\in S_{\ge3}}(-1)^k\in 4\mathbb Z$.
	Moreover, $\sum\limits_{a^k\in S_{\ge 3}}(-1)^{k+1}	=	-\sum\limits_{a^k\in S_{\ge 3}}(-1)^k	\in 4\mathbb Z$.
	
	Finally, for fixed $d\in D_{\ge  3}$ and put $m=\frac nd$. Then $v_2(m)\ge  3$. If $c(h,m)=0$, there is nothing to prove.
	Otherwise,
	\[
	c(h,m)=\mu(t_{m,h})\frac{\varphi(m)}{\varphi(t_{m,h})},
	\quad
	t_{m,h}=\frac{m}{\gcd(m,h)}.
	\]
	Since \(c(h,m)\neq 0\), we have \(\mu(t_{m,h})\neq 0\), and hence \(t_{m,h}\) is square-free.
	Therefore the \(2\)-part of \(t_{m,h}\) is at most \(2\). 
	As the \(2\)-part of \(m\) is at least \(2^3\),
	it follows that $\Xi_h=\sum\limits_{d\in D_{\ge  3}}c(h,\frac nd)\in 4\mathbb Z$.
	Thus every nonzero term \(c(h,m)\) belongs to \(4\mathbb Z\), and therefore $\Xi_h\in 4\mathbb Z$.
	Hence every eigenvalue of \(\Gamma_{\ge 3}\) is divisible by \(4\).
	
	Since \(A_{\ge 3}\) is real symmetric, it is diagonalizable over \(\mathbb R\). 
	Since every eigenvalue \(\lambda\) of \(A_{\ge 3}\) is divisible by \(4\), so $\exp \left(-\frac{\pi\mathrm{i}}2\lambda\right)=1$.
	Therefore, by diagonalizing \(A_{\ge 3}\), we obtain $\exp \left(-\frac{\pi\mathrm{i}}2A_{\ge 3}\right)=I$.
	This completes the proof.
\end{proof}

We are now ready to characterize the case \(8\mid n\).

\begin{thm}\label{thm:8midn-characterization}
	Let \(8\mid n\), \(G\cong D_{2n}\), and let \(S=S_0\cup S_1\subseteq G\setminus\{1\}\) be conjugation-closed, where \(S_0=S\cap\langle a\rangle\) and \(S_1=S\cap b\langle a\rangle\).
	Let \(\Gamma=\mathrm{Cay}(G,S)\) be connected.
	Then \(\Gamma\) admits perfect state transfer if and only if there exist $E\subseteq\{d:d\mid n,\ v_2(\frac nd)=2\}
	\setminus\{\frac n4\}$, $F\subseteq\{d:d\mid n,\ v_2(\frac nd)\ge 3\}$, and $\xi\in\{\frac n4,\frac n2\}$ such that $S_0=
	\bigcup\limits_{d\in D}G_n(d)$, where $D= F\cup E\cup 2E\cup 4E\cup\{\xi\}$.
	In this case, the only restriction on \(S_1\) is that it is nonempty.
\end{thm}
\begin{proof}
\noindent\emph{Necessity.} 
The asserted form of \(D\) follows from Propositions~\ref{pro:D1-D2-8midn} and~\ref{pro:structure-8midn} by taking \(E=D_2\setminus\{\frac n4\}\) and \(F=D_{\ge3}\).
By Corollary~\ref{core:eigenvalues of D2n}, \(S_1\) affects only
\(\lambda_1,\lambda_2,\lambda_3,\lambda_4\). 
Hence, once \(S_0\) has the stated form, it is enough to verify the \(2\)-adic conditions in Lemma~\ref{lem:iff of PST} involving these four eigenvalues for every nonempty conjugation-closed choice of \(S_1\).

Write $A=|S_0|$ and $T=\sum\limits_{a^k\in S_0}(-1)^k$.
For \(e\in E\), put \(m_e=\frac{n}{4e}\). 
Then \(m_e\) is odd. 
The parts \(G_n(e)\), \(G_n(2e)\), and \(G_n(4e)\) contribute to \(A\) is $\varphi(4m_e)+\varphi(2m_e)+\varphi(m_e)=4\varphi(m_e)$.
Since \(8\mid n\), \(e\) is even.
Hence all exponents occurring in \(G_n(e)\), \(G_n(2e)\), \(G_n(4e)\) are even, and so their contribution to \(T\) is also \(4\varphi(m_e)\).
Thus each triple \(\{e,2e,4e\}\) contributes \(0\pmod 4\) to both \(A\) and \(T\).

If \(d\in F\), then \(v_2(\frac nd)\ge 3\), and hence
\(\varphi(\frac nd)\in 4\mathbb Z\). 
Thus \(G_n(d)\) contributes a multiple of \(4\) to \(A\). Moreover, for \(a^k\in G_n(d)\), the integers \(k\) and \(d\) have the same parity. 
Hence \((-1)^k\) is constant on \(G_n(d)\), so the contribution of \(G_n(d)\) to \(T\) is \(\pm\varphi(\frac nd)\), also a multiple of \(4\).

If \(\xi=\frac n4\), then $|G_n(\frac n4)|=\varphi(4)=2$.
Since \(\frac n4\) is even, every exponent occurring in \(G_n(\frac n4)\) is even. 
Thus this part contributes \(2\) to both \(A\) and \(T\). 
If \(\xi=\frac n2\), then $|G_n(\frac n2)|=\varphi(2)=1$, and the same parity argument shows that it contributes \(1\) to both \(A\) and \(T\). 
Therefore, \(A\equiv T\equiv 2\pmod 4\) when \(\xi=\frac n4\), and \(A\equiv T\equiv 1\pmod 4\) when \(\xi=\frac n2\).

By Proposition~\ref{pro:structure-8midn}, \(\eta_1=0\) if \(\xi=\frac n4\), and \(\eta_1=-1\) if \(\xi=\frac n2\).
Since \(8\mid n\), by Corollary~\ref{core:eigenvalues of D2n}, a direct computation gives \(\lambda_1\equiv \lambda_2\equiv A\pmod 4\) and \(\lambda_3\equiv \lambda_4\equiv T\pmod 4\). 
Since \(A\equiv T\pmod 4\), it follows that \(\lambda_1\equiv\lambda_2\equiv\lambda_3\equiv\lambda_4\pmod 4\), and in both cases \(\lambda_1-\eta_1\equiv 2\pmod 4\).
Thus \(v_2(\lambda_1-\eta_1)=1=\gamma\) and \(v_2(\lambda_1-\lambda_i)>1=\gamma\) for \(i=2,3,4\). 
Therefore Lemma~\ref{lem:iff of PST} imposes no restriction on \(\alpha,\beta\).
Hence \(S_1\) is arbitrary.

\medskip
\noindent\emph{Sufficiency.}
Assume \(S_0\) has the stated form and \(S_1\) is a nonempty conjugation-closed subset of \(b\langle a\rangle\). 
We prove that the conditions in Lemma \ref{lem:iff of PST} hold with \(\gamma=1\).
For \(D\), we have $D_{\ge 3}=F$.
Let
\[
S_{\ge 3}:=\bigcup_{d\in F}G_n(d),~~S_0':=\bigcup_{d\in D\setminus F}G_n(d),~~S':=S_0'\cup S_1,~~\Gamma_{\ge3}:=\mathrm{Cay}(G,S_{\ge3}),~~\Gamma':=\mathrm{Cay}(G,S'),\]
and let \(A,A'\), and \(A_{\ge 3}\) be the adjacency matrices of
\(\Gamma,\Gamma'\), and \(\Gamma_{\ge 3}\), respectively. 
Then $A=A'+A_{\ge 3}$.
Since \(S'\) and \(S_{\ge 3}\) are inverse-closed and conjugation-closed, the matrices \(A'\) and \(A_{\ge 3}\) commute. 
By Lemma~\ref{lem:D3 PST}, $\exp(-\frac{\pi\mathrm{i}}2A_{\ge 3})=I$.
Therefore
\begin{align*}
	\exp\left(-\frac{\pi\mathrm{i}}2A\right)=\exp\left(-\frac{\pi\mathrm{i}}2(A'+A_{\ge 3})\right)=\exp\left(-\frac{\pi\mathrm{i}}2A'\right)	\exp\left(-\frac{\pi\mathrm{i}}2A_{\ge3}\right)=\exp\left(-\frac{\pi\mathrm{i}}2A'\right).
\end{align*}
Thus \(\Gamma\) and \(\Gamma'\) have the same transfer matrix at time
\(\frac{\pi}{2}\). 
Hence it is enough to prove that \(\Gamma'\) has perfect state transfer at time \(\frac{\pi}{2}\). 

Let $D'=D\setminus F$.
In the following, we verify Lemma~\ref{lem:iff of PST} for \(\Gamma'\).
We denote the corresponding eigenvalues of \(\Gamma'\) by $\lambda_1'$, $\lambda_2'$, $\lambda_3'$, $\lambda_4'$, $\eta_h'$, $1\le h\le \frac n2-1$.

Write $S_1=\alpha R_0\cup \beta R_1$, $\alpha,\beta\in\{0,1\}$.
We verify Lemma~\ref{lem:iff of PST} with \(\gamma=1\).

For each \(e\in E\), let $m_e=\frac n{4e}$.
Then \(m_e\) is odd. 
For \(1\le j\le \frac n4\), set \(h=2j-1\).
Then Lemma~\ref{lem:ramanujan-odd} gives $c(h,4m_e)+c(h,2m_e)+c(h,m_e)=0$.
Moreover, for \(1\le j\le \frac n4-1\), $c(2j,4m_e)+c(2j,2m_e)+c(2j,m_e)=2(1+(-1)^j)c(2j,m_e)\in 4\mathbb Z$.
Thus every triple \(\{e,2e,4e\}\), \(e\in E\), contributes \(0\) to the
odd-indexed \(\eta'_h\)'s and contributes a multiple of \(4\) to the
even-indexed \(\eta'_{2j}\)'s.
For \(\xi\), we have \(c(2j-1,4)=0\) and \(c(2j,4)\equiv 2\pmod 4\) when
\(\xi=\frac n4\), while \(c(2j-1,2)=-1\) and \(c(2j,2)=1\) when \(\xi=\frac n2\).
Thus
\[
\eta'_{2j-1}=
\begin{cases}
	0, & \xi=\frac n4,\\
	-1, & \xi=\frac n2,
\end{cases}
\quad 1\le j\le \frac n4,\quad \text{and}\quad
\eta'_{2j}\equiv
\begin{cases}
	2, & \xi=\frac n4,\\
	1, & \xi=\frac n2
\end{cases}
\pmod 4,
\quad 1\le j\le \frac n4-1.
\]
Therefore \(v_2(\eta'_{h+1}-\eta'_h)=1\), \(1\le h\le \frac n2-2\).

It remains to check the conditions involving \(\lambda_1',\lambda_2',\lambda_3',\lambda_4'\). 
Let $A'=|S_0'|$, $T'=\sum\limits_{a^k\in S_0'}(-1)^k$.
For each \(e\in E\), the triple \(\{e,2e,4e\}\) contributes $\varphi(4m_e)+\varphi(2m_e)+\varphi(m_e)=4\varphi(m_e)$ to \(A'\), and its contribution to \(T'\) is also a multiple of \(4\).
Therefore, modulo \(4\), the values of \(A'\) and \(T'\) are determined by the contribution of \(\xi\).
Hence
\[
A'\equiv T'\equiv
\begin{cases}
	2, & \xi=\frac n4,\\
	1, & \xi=\frac n2
\end{cases}
\pmod 4.
\]
Since \(8\mid n\), the reflection contributions to
\(\lambda'_1,\lambda'_2,\lambda'_3,\lambda'_4\) are divisible by \(4\).
By Corollary~\ref{core:eigenvalues of D2n}, we obtain $\lambda'_1\equiv \lambda'_2\equiv \lambda'_3\equiv \lambda'_4\pmod 4$.
Thus \(v_2(\lambda'_1-\lambda'_i)>1\), \(i=2,3,4\). 
Moreover, from the above values of \(\eta'_1\), we have $\lambda'_1-\eta'_1\equiv 2 \pmod 4$.
Hence \(v_2(\lambda'_1-\eta'_1)=1\).
All conditions in Lemma~\ref{lem:iff of PST} hold for \(\Gamma'\) with \(\gamma=1\). 
Hence \(\Gamma'\) has perfect state transfer at time \(\frac{\pi}{2}\). 
Since \(\Gamma\) and \(\Gamma'\) have the same transfer matrix at time \(\frac{\pi}{2}\), the graph \(\Gamma\) also has perfect state transfer.
\end{proof}

\begin{core}\label{cor:minimal-time-8midn}
	Under the assumptions of Theorem~\ref{thm:8midn-characterization}, if
	\(\Gamma\) admits perfect state transfer, then the minimum time at which
	perfect state transfer occurs is \(\frac{\pi}{2}\).
\end{core}

\begin{proof}
	Let $M=\gcd\{\lambda_1-\lambda:\lambda\ne \lambda_1\}$.
	By Lemma~\ref{lem:iff of PST}, the minimum perfect state transfer time is $\frac{\pi}{M}$. 
	It remains to show that \(M=2\).
	
	From the proof of Theorem~\ref{thm:8midn-characterization}, all nonzero differences \(\lambda_1-\lambda\) are even. Hence \(2\mid M\).
	We now show that \(M\mid 2\). 
	By Theorem~\ref{thm:8midn-characterization}, $D=F\cup E\cup 2E\cup 4E\cup\{\xi\}$, $\xi\in\{\frac n4,\frac n2\}$.
	The divisors in \(F\) make no contribution to \(\eta_1\) and \(\eta_2\), and each triple \(\{e,2e,4e\}\), with \(e\in E\), also makes no contribution to \(\eta_1\) and \(\eta_2\). 
	Thus only \(\xi\) contributes to \(\eta_1\) and \(\eta_2\). 
	Hence
	\[
	\eta_1=
	\begin{cases}
		0, & \xi=\frac n4,\\
		-1, & \xi=\frac n2,
	\end{cases}
	\quad \text{and}\quad
	\eta_2=
	\begin{cases}
		-2, & \xi=\frac n4,\\
		1, & \xi=\frac n2.
	\end{cases}
	\]
	Therefore $\eta_2-\eta_1=\pm 2$.
	Since \(\Gamma\) is connected, \(\lambda_1\) is the unique largest eigenvalue. 
	Hence \(\eta_1,\eta_2\ne\lambda_1\).
	It follows that $M$ divides both \(\lambda_1-\eta_1\) and \(\lambda_1-\eta_2\), and hence $M\mid (\eta_2-\eta_1)=\pm 2$.
	Therefore \(M\mid 2\). 
	Since \(2\mid M\), we get \(M=2\). 
	Hence the minimum perfect state transfer time is $\frac{\pi}{M}=\frac{\pi}{2}$.
\end{proof}

It remains to consider the case \(n\equiv 4 \pmod 8\). 
Since \(S\) is conjugation-closed, we may write $S_1=\alpha b\langle a^2\rangle\cup \beta ba\langle a^2\rangle$, $\alpha,\beta\in\{0,1\}$.
Since \(\Gamma\) is connected, we have \(S_1\neq\emptyset\). 
Hence the case \(\alpha=\beta\) is equivalent to \(S_1=b\langle a\rangle\), whereas the case \(\alpha\ne\beta\) is equivalent to \(S_1\in\{b\langle a^2\rangle,ba\langle a^2\rangle\}\). 
We now characterize the connection sets in these two cases separately.
Before doing so, we record a lemma that will be used in both arguments.

\begin{lem}\label{lem:n4mod8-common-computation}
	Let \(n=4q\), where \(q\) is odd.
	Assume that \(\Gamma\) admits perfect state transfer. 
	Then \(\Gamma\) is integral, and hence $S_0=\bigcup\limits_{d\in D}G_n(d)$, where \(D\) is a set of proper divisors of \(n\). 
	Set $H=(D_1\setminus\{2q\})\cup D_2$, and let $\sigma_j=\sum\limits_{d\in H}c(j,\frac{2q}{d})$, $0\le j\le 2q-1$, be the eigenvalues of \(\mathrm{ICG}_{2q}(H)\). 
	Let $\varepsilon=1$ if $2q\in D_1$ and $\varepsilon=0$ if $2q\notin D_1$.
	Then $\eta_{2j}=2\sigma_j+\varepsilon$, $\eta_{2j-1}=-\varepsilon$, where \(1\le j\le q-1\).
	Moreover, $|S_0|=2\sigma_0+\varepsilon$, $\sum\limits_{a^k\in S_0}(-1)^k=2\sigma_q+\varepsilon$.
\end{lem}

\begin{proof}
	Since \(\Gamma\) admits perfect state transfer, Lemma~\ref{thm:-1 0} gives $D_0=2(D_1\setminus\{2q\})$.
	As \(n=4q\), we have \(D=D_0\cup D_1\cup D_2\), and hence $	D=(D_1\setminus\{2q\})\cup 2(D_1\setminus\{2q\})\cup D_2\cup \varepsilon\{2q\}$.
	
	Let \(1\le j\le q-1\). 
	The pair \(\{e,2e\}\), \(e\in D_1\setminus\{2q\}\), contributes to $\eta_{2j}$ is $c(2j,\frac ne)+c(2j,\frac n{2e})=c(2j,\frac{4q}{e})+c(2j,\frac{2q}{e})=2c(j,\frac{2q}{e})$.
	The divisor \(d\in D_2\) contributes to $\eta_{2j}$ is $c(2j,\frac nd)=c(2j,\frac{4q}{d})=2c(j,\frac{2q}{d})$.
	The possible divisor \(2q\) contributes \(c(2j,2)=1\). 
	Therefore
	\[
	\eta_{2j}
	=
	2\sum_{e\in D_1\setminus\{2q\}}
	c\left(j,\frac{2q}{e}\right)
	+
	2\sum_{d\in D_2}
	c\left(j,\frac{2q}{d}\right)
	+\varepsilon
	=
	2\sigma_j+\varepsilon.
	\]
	
	Next let \(h=2j-1\), where \(1\le j\le q\). For
	\(e\in D_1\setminus\{2q\}\), put \(m=\frac{2q}{e}\). Then \(m\) is odd, and $c(h,2m)+c(h,m)=0$.
	Thus each pair \(\{e,2e\}\) contributes \(0\) to \(\eta_h\). 
	Each
	\(d\in D_2\) contributes \(c(h,\frac{4q}{d})=0\). The possible divisor
	\(2q\) contributes \(c(h,2)=-1\). 
	Hence $\eta_{2j-1}=-\varepsilon$, $1\le j\le q$.
	
	It remains to compute \(|S_0|\) and \(\sum\limits_{a^k\in S_0}(-1)^k\). 
	For \(|S_0|\), each pair \(\{e,2e\}\), where \(e\in D_1\setminus\{2q\}\), contributes $\varphi(\frac{4q}{e})	+\varphi(\frac{2q}{e})=2\varphi(\frac{2q}{e})$,
	each \(d\in D_2\) contribute $\varphi(\frac{4q}{d})	=2\varphi(\frac{2q}{d})$,
	and the possible divisor \(2q\) contributes \(\varphi(2)=1\). 
	Thus
	\[
|S_0|=2\sum_{s\in H}\varphi\left(\frac{2q}{s}\right)+\varepsilon	=2\sigma_0+\varepsilon.
	\]
	
	For \(\sum\limits_{a^k\in S_0}(-1)^k\), each pair \(\{e,2e\}\), where \(e\in D_1\setminus\{2q\}\), contributes \(2\varphi(\frac{2q}{e})\), while each \(d\in D_2\) contributes \(-2\varphi(\frac{2q}{d})\). 
	The possible divisor \(2q\) contributes \(1\). 
	Since $c(q,\frac{2q}{e})=\varphi(\frac{2q}{e})$ and $c(q,\frac{2q}{d})=-\varphi(\frac{2q}{d})$, for \(e\in D_1\setminus\{2q\}\) and \(d\in D_2\), respectively, we obtain $\sum\limits_{a^k\in S_0}(-1)^k=2\sigma_q+\varepsilon$.
\end{proof}

\begin{thm}\label{thm:n4mod8-full-reflection}
	Let \(n\equiv 4\pmod 8\), \(G\cong D_{2n}\), and let
	\(S=S_0\cup S_1\subseteq G\setminus\{1\}\) be conjugation-closed, where $S_0=S\cap\langle a\rangle$, $S_1=S\cap b\langle a\rangle$.
	Assume that $S_1=b\langle a\rangle$.
	Let $P:=\{d:d\mid n,\ v_2\left(\frac nd\right)=2\}\setminus\{\frac n4\}$.
	Then \(\Gamma=\mathrm{Cay}(G,S)\) admits perfect state transfer if and only if there exist \(E\subseteq P\) and \(\xi\in\left\{\frac n4,\frac n2\right\}\) such that, for
	\(D=E\cup 2E\cup 4E\cup\{\xi\}\),
	$S_0=\bigcup\limits_{d\in D}G_n(d)$.
\end{thm}
\begin{proof}
	Put $q=\frac n4$.	
	Then \(q\) is odd and \(\frac n2=2q\).
	
	If \(q=1\), then \(n=4\) and \(P=\emptyset\). 
	Since \(S_1= b\langle a\rangle\), the possible divisor sets are \(D\subseteq\{1,2\}\). 
	A direct calculation from Corollary \ref{core:eigenvalues of D2n} and Lemma~\ref{lem:iff of PST} gives
	\[
	\begin{array}{c|c|c|c}
		D & \eta_1 & (\lambda_1,\lambda_2,\lambda_3,\lambda_4)
		& \text{PST}\\
		\hline
		\emptyset & 0 & (4,-4,0,0) & \text{no}\\
		\{1\} & 0 & (6,-2,-2,-2) & \text{yes}\\
		\{2\} & -1 & (5,-3,1,1) & \text{yes}\\
		\{1,2\} & -1 & (7,-1,-1,-1) & \text{no}.
	\end{array}
	\]
	Thus \(\Gamma\) admits perfect state transfer if and only if
	\(D=\{1\}\) or \(D=\{2\}\). This is exactly the stated form, since
	\(P=\emptyset\) and \(\xi\in\{1,2\}\). Hence the assertion holds when
	\(q=1\). In the rest of the proof, assume that \(q>1\).

	\medskip
	\noindent\emph{Necessity.}
	Assume that \(\Gamma\) admits perfect state transfer. Then \(\Gamma\) is
	integral, and hence there exists a set \(D\) of proper divisors of \(n\) such 	that $S_0=\bigcup\limits_{d\in D}G_n(d)$.
	Since \(S_1=b\langle a\rangle\), we may write $\alpha=\beta=1$.
	By \(n=4q\) and Lemma~\ref{thm:-1 0}, we have \(D=D_0\cup D_1\cup D_2=2(D_1\setminus\{2q\})\cup D_1\cup D_2\). 
	Set $H:=(D_1\setminus\{2q\})\cup D_2$, and let $\sigma_j=\sum\limits_{d\in H}c(j,\frac{2q}{d})$, $0\le j\le 2q-1$, be the eigenvalues of \(\mathrm{ICG}_{2q}(H)\). 
	Let $\varepsilon=1$ if $2q\in D_1$ and $\varepsilon=0$ if $2q\notin D_1$.
	
	We prove that \(D\) has the required form in two steps. 
	First, we show that all odd-position eigenvalues of \(\mathrm{ICG}_{2q}(H)\) have the same parity.
	This determines the possible structure of \(H\). 
	Second, we use the condition on \(\eta_2-\eta_1\) to decide which one of the two special divisors \(q\) and \(2q\) occurs.
	
	By Lemma \ref{lem:n4mod8-common-computation}, for odd \(j\) with \(1\le j\le q-1\), $\eta_{2j}=2\sigma_j+\varepsilon$.
    By Lemma \ref{thm:-1 0}, $\eta_{2j-1}=-\varepsilon$.
    Thus Lemma~\ref{lem:iff of PST} implies $v_2(\eta_{2j}-\eta_{2j-1})=v_2(2(\sigma_j+\varepsilon))=\gamma$, and hence $v_2(\sigma_j+\varepsilon)=\gamma-1$.
    
    Now let $A=|S_0|$, $T=\sum\limits_{a^k\in S_0}(-1)^k$.
	By Lemma \ref{lem:n4mod8-common-computation}, $A=2\sigma_0+\varepsilon$, and $T=2\sigma_q+\varepsilon$.
	By Corollary~\ref{core:eigenvalues of D2n}, $\lambda_1=2\sigma_0+\varepsilon+4q$, and $\lambda_3=\lambda_4=2\sigma_q+\varepsilon$.	
	Also Lemma \ref{thm:-1 0} gives $\eta_1=-\varepsilon$.
	Thus Lemma~\ref{lem:iff of PST} gives $v_2(\lambda_1-\eta_1)=v_2(\lambda_1+\varepsilon)=\gamma$.
	Equivalently, $v_2(\sigma_0+\varepsilon+2q)=\gamma-1$.
	Moreover, from \(v_2(\lambda_1-\lambda_3)>\gamma\), we get $v_2(\sigma_0-\sigma_q+2q)>\gamma-1$.
	Therefore the difference of these two numbers has \(2\)-adic valuation \(\gamma-1\), and hence $v_2(\sigma_q+\varepsilon)=\gamma-1$.
		Combining this with $v_2(\sigma_j+\varepsilon)=\gamma-1$, where $1\le j\le q-1$ is odd, we see that \(\sigma_j+\varepsilon\) has the same parity for every odd
	\(j\) with \(1\le j\le q\). 
	
	Since \(\mathrm{ICG}_{2q}(H)\) is undirected, \(\sigma_{2q-j}=\sigma_j\). Hence all
	odd-position eigenvalues of \(\mathrm{ICG}_{2q}(H)\) have the same parity.
	By Lemma~\ref{lem:basic lem12+13}, either $H=D_2\cup 2D_2$, or $H=(D_2\setminus\{q\})\cup 2(D_2\setminus\{q\})\cup\{q\}$.
	
	If \(H=D_2\cup 2D_2\), then \(q\notin D_2\).
	Comparing this with $H=(D_1\setminus\{2q\})\cup D_2$ gives $D_1\setminus\{2q\}=2D_2$, $D_0=4D_2$.	 
	The contributions of each pair \(\{d,2d\}\), \(d\in D_2\), to \(\sigma_1\) cancel, and so \(\sigma_1=0\).
	Since \(\eta_2=2\sigma_1+\varepsilon\) and \(\eta_1=-\varepsilon\), we get $\eta_2-\eta_1=2\varepsilon$.
	Since \(v_2(\eta_2-\eta_1)=\gamma\) is finite, we have  \(\varepsilon=1\). 
	Hence $D=D_0\cup D_1\cup D_2=D_2\cup 2D_2\cup 4D_2\cup\{2q\}$.
	This has the required form, with $E=D_2$, $\xi=2q=\frac n2$.
	
	If $H=(D_2\setminus\{q\})\cup 2(D_2\setminus\{q\})\cup\{q\}$, then \(q\in D_2\).
	The contributions of each pair \(\{d,2d\}\), \(d\in D_2\setminus\{q\}\), to
	 \(\sigma_1\) cancel, while the divisor \(q\) contributes \(c(1,2)=-1\).
	 Thus \(\sigma_1=-1\).
	 Since \(\eta_2=2\sigma_1+\varepsilon\) and $\eta_1=-\varepsilon$, we get $\eta_2-\eta_1=2(\varepsilon-1)$.
	 Since \(v_2(\eta_2-\eta_1)=\gamma\) is finite, we have \(\varepsilon=0\).
	Hence $D=(D_2\setminus\{q\})\cup 2(D_2\setminus\{q\})\cup 4(D_2\setminus\{q\})\cup\{q\}$.
	This has the required form, with $E=D_2\setminus\{q\}$, $\xi=q=\frac n4$.
	
	\medskip
	\noindent\emph{Sufficiency.}
	Conversely, assume that the condition in the statement holds. 
	Since \(S_1=b\langle a\rangle\), we have $\alpha=\beta=1$.
	We verify Lemma~\ref{lem:iff of PST} with \(\gamma=1\).
	
For each \(e\in E\), put \(m_e=\frac qe\). 
Then $\frac n e=4m_e$, $\frac n{2e}=2m_e$, $\frac n{4e}=m_e$.
The same calculation as in the sufficiency part of Theorem~\ref{thm:8midn-characterization} shows that, for $1\le j\le q$, each triple \(\{e,2e,4e\}\), \(e\in E\), contributes \(0\) to \(\eta_{2j-1}\), and that, for $1\le j\le q-1$, it contributes a multiple of \(4\) to \(\eta_{2j}\).
For \(\xi\), we have \(c(2j-1,4)=0\) and \(c(2j,4)\equiv 2\pmod 4\) when \(\xi=q\), while \(c(2j-1,2)=-1\) and \(c(2j,2)=1\) when \(\xi=2q\). 
Thus
\[
\eta_{2j-1}=
\begin{cases}
	0, & \xi=q,\\
	-1, & \xi=2q,
\end{cases}
\quad 1\le j\le q,
\quad \text{and}\quad
\eta_{2j}\equiv
\begin{cases}
	2, & \xi=q,\\
	1, & \xi=2q
\end{cases}
\pmod 4,
\quad 1\le j\le q-1.
\]
Therefore \(v_2(\eta_{h+1}-\eta_h)=1\), \(1\le h\le 2q-2\).

We next check the conditions involving \(\lambda_1,\lambda_2,\lambda_3,\lambda_4\).
Let $A=|S_0|$, $T=\sum\limits_{a^k\in S_0}(-1)^k$.
As in the proof of Theorem~\ref{thm:8midn-characterization}, each triple
\(\{e,2e,4e\}\), \(e\in E\), contributes a multiple of \(4\) to both
\(A\) and \(T\). Therefore, modulo \(4\), the values of \(A\) and \(T\)
are determined by the contribution of \(\xi\). 
Hence
\[
A\equiv T\equiv
\begin{cases}
	2, & \xi=q,\\
	1, & \xi=2q
\end{cases}
\pmod 4.
\]
Since \(S_1=b\langle a\rangle\), the contributions from $S_1$ to \(\lambda_1,\lambda_2,\lambda_3,\lambda_4\) are \(4q,-4q,0,0\), all of which are divisible by \(4\). 
Therefore $\lambda_1\equiv\lambda_2\equiv\lambda_3\equiv\lambda_4\pmod 4$.
Thus \(v_2(\lambda_1-\lambda_i)>1\), \(i=2,3,4\). Moreover, from the above
values of \(\eta_1\), we have $\lambda_1-\eta_1\equiv 2\pmod 4$.
Hence \(v_2(\lambda_1-\eta_1)=1\).
Thus all conditions in Lemma~\ref{lem:iff of PST} hold with \(\gamma=1\).
Therefore \(\Gamma\) admits perfect state transfer.
\end{proof}

\begin{thm}\label{thm:n4mod8-half-reflection}
	Let \(n\equiv 4\pmod 8\), \(G\cong D_{2n}\), and let
	\(S=S_0\cup S_1\subseteq G\setminus\{1\}\) be conjugation-closed, where $S_0=S\cap\langle a\rangle$, $S_1=S\cap b\langle a\rangle$.
	Assume that $S_1\in\{b\langle a^2\rangle,ba\langle a^2\rangle\}$ and  \(\Gamma=\mathrm{Cay}(G,S)\) is connected.
	Let $P:=\{d:d\mid n,\ v_2(\frac nd)=2\}\setminus\{\frac n4\}$.
	Then \(\Gamma\) admits perfect state transfer if and only if
	there exist \(E\subseteq P\) and \(\rho\in\{0,1\}\) with \(E\ne\emptyset\) or \(\rho=1\) such that, for $D=E\cup 2(P\setminus E)\cup 4(P\setminus E)
	\cup \rho\{\frac n4,\frac n2\}$, $S_0=\bigcup\limits_{d\in D}G_n(d)$.
\end{thm}
\begin{proof}
	Put \(q=\frac n4\). 
	Then \(q\) is odd, \(\frac n2=2q\), and \(P\) is the set
	of proper divisors of \(q\).
	
	If \(q=1\), then \(n=4\) and \(P=\emptyset\). Since
	\(S_1\in\{b\langle a^2\rangle,ba\langle a^2\rangle\}\), write
	\(\alpha+\beta=1\) and put \(s=\alpha-\beta\in\{1,-1\}\). The possible
	divisor sets are \(D\subseteq\{1,2\}\). 
A direct calculation from Corollary \ref{core:eigenvalues of D2n} and Lemma~\ref{lem:iff of PST} gives
	\[
	\begin{array}{c|c|c|c}
		D & \eta_1 & (\lambda_1,\lambda_2,\lambda_3,\lambda_4)
		& \text{PST}\\
		\hline
		\emptyset & 0 & (2,-2,2s,-2s) & \text{yes}\\
		\{1\} & 0 & (4,0,-2+2s,-2-2s) & \text{no}\\
		\{2\} & -1 & (3,-1,1+2s,1-2s) & \text{no}\\
		\{1,2\} & -1 & (5,1,-1+2s,-1-2s) & \text{yes}.
	\end{array}
	\]
	If $D=\emptyset$, then $\Gamma$ is not connected.
	Thus \(D=\{1,2\}\). 
	This is exactly the stated form since \(P=\emptyset\) and \(\rho=1\). 
	Hence the assertion holds when \(q=1\). 
	In the rest of the proof, assume that \(q>1\).

	\medskip
	\noindent\emph{Necessity.}
	Assume that \(\Gamma\) admits perfect state transfer. 
	Then \(\Gamma\) is integral, and hence there exists a set \(D\) of proper divisors of \(n\) such that $S_0=\bigcup\limits_{d\in D}G_n(d)$.
	
	We determine \(D\) in two steps. 
	First, we prove that the two special divisors \(q\) and \(2q\) 	either both occur or both do not occur. 
	Second, we prove that, for each \(p\in P\), exactly one of \(p\in D_2\) and \(2p\in D_1\) holds. 
	These two facts will give the stated description of	\(D\).

	Since \(S_1\in\{b\langle a^2\rangle,ba\langle a^2\rangle\}\), we may write $S_1=\alpha b\langle a^2\rangle\cup \beta ba\langle a^2\rangle$, $\alpha,\beta\in\{0,1\}$, where \(\alpha\ne\beta\). 
	Thus \(\alpha+\beta=1\).
	By \(n=4q\) and Lemma~\ref{thm:-1 0}, we have \(D=D_0\cup D_1\cup D_2=2(D_1\setminus\{2q\})\cup D_1\cup D_2\). 
	Set $H=(D_1\setminus\{2q\})\cup D_2$, and let $\sigma_j=\sum\limits_{d\in H}c(j,\frac{2q}{d})$, $0\le j\le 2q-1$, be the eigenvalues of \(\mathrm{ICG}_{2q}(H)\). 
	Also let 
	\[
	\varepsilon=
	\begin{cases}
		1, & 2q\in D_1,\\
		0, & 2q\notin D_1,
	\end{cases}
	\quad \text{and}\quad
	\delta=
	\begin{cases}
		1, & q\in D_2,\\
		0, & q\notin D_2.
	\end{cases}
	\]
	
	By Lemma \ref{lem:n4mod8-common-computation}, for every odd $j$ with \(1\le j\le q-1\), $\eta_{2j}=2\sigma_j+\varepsilon$.
	Moreover, Lemma~\ref{thm:-1 0} gives $\eta_{2j-1}=-\varepsilon$.
	Thus Lemma~\ref{lem:iff of PST} gives $v_2(\eta_{2j}-\eta_{2j-1})=v_2(2(\sigma_j+\varepsilon))=\gamma$,
	and therefore $v_2(\sigma_j+\varepsilon)=\gamma-1$.
	
	Now let $A=|S_0|$, $T=\sum\limits_{a^k\in S_0}(-1)^k$.
	By Lemma \ref{lem:n4mod8-common-computation}, $A=2\sigma_0+\varepsilon$ and $T=2\sigma_q+\varepsilon$.
	By Corollary~\ref{core:eigenvalues of D2n}, $\lambda_1=2\sigma_0+\varepsilon+2q(\alpha+\beta)$, and $\lambda_3=2\sigma_q+\varepsilon+2q(\alpha-\beta)$.
	By Lemma \ref{thm:-1 0}, we have $\eta_1=-\varepsilon$.
	Hence Lemma~\ref{lem:iff of PST} gives $v_2(\lambda_1-\eta_1)=v_2(\lambda_1+\varepsilon)=\gamma$.
	Equivalently, $v_2(\sigma_0+\varepsilon+q(\alpha+\beta))=\gamma-1$.
	Also, from \(v_2(\lambda_1-\lambda_3)>\gamma\), we obtain $v_2(\sigma_0-\sigma_q+2q\beta)>\gamma-1$.	
	Taking the difference of the two numbers above gives $v_2(\sigma_q+\varepsilon+q(\alpha-\beta))=\gamma-1$.
		
	Since \(\alpha+\beta=1\), we have $\lambda_1-\lambda_2=n=4q$.
	By Lemma~\ref{lem:iff of PST}, \(v_2(\lambda_1-\lambda_2)>\gamma\), and hence \(\gamma<2\). 
	On the other hand, \(\lambda_1-\eta_1=\lambda_1+\varepsilon\) is even, so \(\gamma\ge 1\). 
	Therefore $\gamma=1$.
	Thus \(\sigma_j+\varepsilon\) is odd for every odd \(j\) with \(1\le j\le q-1\). Since \(q\) is odd and \(\alpha-\beta=\pm1\), the relation $v_2(\sigma_q+\varepsilon+q(\alpha-\beta))=0$ implies that \(\sigma_q+\varepsilon\) is even. 
	Hence \(\sigma_q\) has the opposite parity to those \(\sigma_j\)'s.
	
	We now determine the relation between the two divisors \(q\) and \(2q\).
	Recall that $H=(D_1\setminus\{2q\})\cup D_2$.
	The possible divisor \(q\in D_2\) will be treated separately. 
	Every other element of \(D_2\) lies in \(P\), and every element of \(D_1\setminus\{2q\}\) has the form \(2p\) with \(p\in P\).
	
	For \(p\in P\), put \(m_p=\frac qp\). 
	Then \(m_p>1\) is odd. 
	If \(p\in D_2\), its contribution to \(\sigma_q\) is $c(q,\frac{2q}{p})	=c(q,2m_p)=-\varphi(m_p)$, which is even. 
	If \(2p\in D_1\), its contribution is $c(q,\frac{2q}{2p})=c(q,m_p)=\varphi(m_p)$, which is also even. 
	Thus every contribution to \(\sigma_q\) other than the possible one from \(q\in D_2\) is even. 
	The divisor \(q\), if it occurs, contributes \(c(q,2)=-1\). 
	Hence $\sigma_q\equiv \delta\pmod 2$.
	Since \(\sigma_q+\varepsilon\) is even, it follows that $\delta=\varepsilon$.
	Thus \(q\in D_2\) if and only if \(2q\in D_1\).
	
	It remains to prove the corresponding statement for all divisors in \(P\).
	Define $B:=\{p\in P:\text{ exactly one of }p\in D_2\text{ and }2p\in D_1
	\text{ holds}\}$.
	We shall prove that \(B=P\).
	
	For \(1\le j\le q-1\), put $\widetilde{\sigma}_j:=\sigma_j-\delta c(j,2)$.
	Thus \(\widetilde{\sigma}_j\) is obtained from \(\sigma_j\) by removing the
	possible contribution of the divisor \(q\in D_2\). 
	Since \(\delta=\varepsilon\) and \(c(j,2)\) is odd, the fact that \(\sigma_j+\varepsilon\) is odd for odd \(j\) implies that $\widetilde{\sigma}_j\equiv 1\pmod 2$ for every odd \(j\) with \(1\le j\le q-1\).
	
	We compute \(\widetilde{\sigma}_j\) modulo \(2\). 
	After removing the possible divisor \(q\), every remaining element of \(H\) has the form \(p\) or \(2p\) for some \(p\in P\). 
	Fix \(p\in P\), and put \(m_p=\frac qp\). 
	Then \(m_p\) is odd. 
	If \(p\in D_2\), its contribution to \(\widetilde{\sigma}_j\) is $c(j,\frac{2q}{p})=c(j,2m_p)$.
	If \(2p\in D_1\), its contribution is $c(j,\frac{2q}{2p})=c(j,m_p)$.
	Since \(m_p\) is odd, Lemma \ref{lem:ramanujan-odd} gives $c(j,2m_p)=c(j,2)c(j,m_p)\equiv c(j,m_p)\pmod 2$.
	Therefore, modulo \(2\), the pair \(\{p,2p\}\) contributes \(c(j,m_p)\) if exactly one of \(p\in D_2\) and \(2p\in D_1\) holds, and contributes \(0\) otherwise. 
	Hence $\widetilde{\sigma}_j\equiv\sum\limits_{p\in B}c(j,\frac qp)\pmod 2$.
	Thus for every odd \(j\) with \(1\le j\le q-1\), $\sum\limits_{p\in B}c(j,\frac qp)\equiv 1\pmod 2$.
	
	Next we remove the restriction that \(j\) is odd. 
	Let \(1\le j\le q-1\) be arbitrary and write \(j=2^r u\), where \(u\) is odd. 
	Since \(q\) is odd, \(\frac qp\) is odd for every \(p\in P\), and hence $\gcd(j,\frac qp)=\gcd(u,\frac qp)$.
	Therefore $c(j,\frac qp)=c(u,\frac qp)$.
	The congruence above now gives for every \(1\le j\le q-1\), $\sum\limits_{p\in B}c(j,\frac qp)\equiv 1\pmod 2$.
		
	Let $\rho_j:=\sum\limits_{p\in B}c(j,\frac qp)$, $0\le j\le q-1$.
	Then \(\rho_0,\rho_1,\ldots,\rho_{q-1}\) are the eigenvalues of \(\mathrm{ICG}_q(B)\), and every nontrivial eigenvalue \(\rho_j\) is odd. 
	By Lemma~\ref{lem:auxiliary-Km}, we obtain $B=P$.
	Therefore, for every \(p\in P\), exactly one of \(p\in D_2\) and
	\(2p\in D_1\) holds.
	
	We now finish the description of \(D\). 
	Set $E=\{p\in P:p\in D_2\}$.
	Since \(B=P\), we have $D_2\setminus\{q\}=E$, $D_1\setminus\{2q\}=2(P\setminus E)$.
	By \(D_0=2(D_1\setminus\{2q\})\), we get $D_0=4(P\setminus E)$.
	Together with \(\delta=\varepsilon\), this gives $D=E\cup 2(P\setminus E)\cup 4(P\setminus E)\cup \rho\{q,2q\}$, where \(\rho=\delta=\varepsilon\). 
	
	Finally, since \(\Gamma\) is connected, \(\langle S\rangle=G\).
	If \(D\) contains no odd divisor, then every element of \(S_0\) is an even power of \(a\), that is, \(S_0\subseteq \langle a^2\rangle\). 
	Together with \(S_1\in\{b\langle a^2\rangle,ba\langle a^2\rangle\}\), we have $\langle S\rangle\neq G$, a contradiction.
	Hence \(D\) must contain an odd divisor. 
	In the description above, this is equivalent to \(E\ne\emptyset\) or \(\rho=1\). 
	This completes the proof of necessity.
	
	\medskip
	\noindent\emph{Sufficiency.}
	Conversely, assume that the stated condition holds. 
	In the notation \(S_1=\alpha b\langle a^2\rangle\cup \beta ba\langle a^2\rangle\),
	\(\alpha,\beta\in\{0,1\}\), we have \(\alpha+\beta=1\).
	We verify Lemma~\ref{lem:iff of PST} with \(\gamma=1\).
	
	For \(p\in P\), put \(m_p=\frac qp\). 
	Then \(m_p>1\) is odd and $\frac np=4m_p$, $\frac n{2p}=2m_p$, $\frac n{4p}=m_p$.
	Let \(h=2j-1\), where \(1\le j\le q\). 
	If \(p\in E\), then \(c(h,4m_p)=0\). 
	If \(p\in P\setminus E\), then $c(h,2m_p)+c(h,m_p)=0$.
	Hence the divisors outside \(\{q,2q\}\) make no contribution to \(\eta_{2j-1}\). 
	The possible divisors \(q\) and \(2q\) contribute $c(2j-1,4)+c(2j-1,2)=-1$.
	Thus $\eta_{2j-1}=-\rho$, $1\le j\le q$.
	
	Next let \(1\le j\le q-1\). 
	For \(p\in E\), $c(2j,4m_p)\equiv 2c(2j,m_p)\pmod 4$, while for \(p\in P\setminus E\), $c(2j,2m_p)+c(2j,m_p)=2c(2j,m_p)$.
	Therefore the divisors outside \(\{q,2q\}\) contribute $2\sum\limits_{p\in P}c(2j,\frac qp)$ to \(\eta_{2j}\) modulo \(4\). 
	Since \(P\) is the set of all proper divisors of \(q\), this sum is a nontrivial eigenvalue of \(K_q\), and hence equals \(-1\). 
	Thus these divisors contribute \(2\pmod 4\).
	
	The possible divisors \(q\) and \(2q\) contribute $c(2j,4)+c(2j,2)=2(-1)^j+1\equiv 3\pmod 4$.
	Hence $\eta_{2j}\equiv 2+3\rho\pmod 4$, $1\le j\le q-1$.
	Together with \(\eta_{2j-1}=-\rho\), this gives $v_2(\eta_{h+1}-\eta_h)=1,\quad 1\le h\le 2q-2$.
	
	We next check the conditions involving \(\lambda_1,\lambda_2,\lambda_3,\lambda_4\).
	Let $A=|S_0|$, $T=\sum\limits_{a^k\in S_0}(-1)^k$.
	For \(p\in E\), the divisor \(p\) contributes \(2\varphi(m_p)\) to \(A\) and \(-2\varphi(m_p)\) to \(T\). 
	For \(p\in P\setminus E\), the divisors \(2p\) and \(4p\) contribute \(2\varphi(m_p)\) to both \(A\) and \(T\).
	Since \(m_p>1\) is odd, \(\varphi(m_p)\) is even. 
	Hence all these contributions are divisible by \(4\).
	
	The possible divisors \(q\) and \(2q\) contribute \(3\) to \(A\) and \(-1\equiv 3\pmod 4\) to \(T\). Hence $A\equiv T\equiv 3\rho\pmod 4$.
	Since \(\alpha+\beta=1\), and since \(2q\equiv -2q\equiv 2\pmod 4\),
	Corollary~\ref{core:eigenvalues of D2n} gives $\lambda_1\equiv\lambda_2\equiv\lambda_3\equiv\lambda_4
	\equiv 3\rho+2\pmod 4$.
	Thus \(v_2(\lambda_1-\lambda_i)>1\), \(i=2,3,4\). 
	Moreover, from \(\eta_1=-\rho\), we have $\lambda_1-\eta_1\equiv (3\rho+2)-(-\rho)\equiv 2\pmod 4$.
	Hence \(v_2(\lambda_1-\eta_1)=1\). 
	Therefore all conditions in Lemma~\ref{lem:iff of PST} hold with \(\gamma=1\), and so \(\Gamma\) admits perfect state transfer.
\end{proof}

\begin{core}\label{cor:minimal-time-n4mod8}
	Under the hypotheses of Theorems~\ref{thm:n4mod8-full-reflection}
	and~\ref{thm:n4mod8-half-reflection}, whenever \(\Gamma\) admits perfect
	state transfer, the minimum perfect state transfer time is \(\frac{\pi}{2}\).
\end{core}
\begin{proof}
	By Lemma~\ref{lem:iff of PST}, the minimum perfect state transfer time is $\frac{\pi}{M}$, $M=\gcd\{\lambda_1-\lambda:\lambda\neq \lambda_1\}$.
	It suffices to prove \(M=2\).
	Let \(n=4q\), where \(q\) is odd. 
	If \(q=1\), then the direct computations in the proofs of Theorems~\ref{thm:n4mod8-full-reflection}	and~\ref{thm:n4mod8-half-reflection} show that every admissible case has \(M=2\).
	Hence we may assume that \(q>1\).

	In both characterizations, all eigenvalue differences \(\lambda_1-\lambda\) are
	even. 
	Hence \(2\mid M\). We show that \(M\mid 2\).
	
	First assume that \(S_1=b\langle a\rangle\). 
	By Theorem~\ref{thm:n4mod8-full-reflection}, $D=E\cup 2E\cup 4E\cup\{\xi\}$, $\xi\in\{q,2q\}$.
	Each triple \(\{e,2e,4e\}\), \(e\in E\), contribute \(0\) to both \(\eta_1\) and \(\eta_2\). 
	The divisor \(\xi\) gives $(\eta_1,\eta_2)=(0,-2)$ if $\xi=q$, and $(\eta_1,\eta_2)=(-1,1)$ if $\xi=2q$.
	Hence \(\eta_2-\eta_1=\pm2\). 
	Since $\Gamma$ is connected, $\lambda_1$ is the unique largest eigenvalue, and so $\eta_1$, $\eta_2\neq \lambda_1$.
	Therefore \(M\) divides both \(\lambda_1-\eta_1\) and \(\lambda_1-\eta_2\), and so \(M\mid(\eta_2-\eta_1)=\pm 2\).
    Thus \(M=2\).
	
	Now assume that \(S_1\in\{b\langle a^2\rangle,ba\langle a^2\rangle\}\). 
	By Theorem~\ref{thm:n4mod8-half-reflection}, $D=E\cup 2(P\setminus E)\cup 4(P\setminus E)\cup \rho\{q,2q\}$, $\rho\in\{0,1\}$.
	Here \(S_1=\alpha b\langle a^2\rangle\cup\beta ba\langle a^2\rangle\) with \(\alpha+\beta=1\). 
	For every \(p\in P\), the corresponding contribution to \(|S_0|\) is \(2\varphi(\frac qp)\), and hence $|S_0|=2(q-1)+3\rho$.
	Thus $\lambda_1=|S_0|+2q=4q-2+3\rho$, $\eta_1=-\rho$.
	Therefore $\lambda_1-\eta_1=2(2q-1+2\rho)$.
	Moreover, $\lambda_1-\lambda_2=4q$.
	Hence $	M\mid \gcd(4q,2(2q-1+2\rho))=2$, since \(\gcd(2q,2q-1+2\rho)=1\) for \(\rho\in\{0,1\}\). 
	Thus \(M=2\) in this case as well.
	
	Consequently, the minimum perfect state transfer time is $\frac{\pi}{M}=\frac{\pi}{2}$.
	\end{proof}
\begin{rem}\label{rem:1}
Theorems \ref{thm:8midn-characterization}, \ref{thm:n4mod8-full-reflection} and \ref{thm:n4mod8-half-reflection} are arguably one of the most practical contributions of this paper. 
Although Lemma \ref{lem:iff of PST} \cite[Theorem 3.2]{Perfect state transfer on Cayley graphs over dihedral group} gives an eigenvalue criterion for perfect state transfer, it is not easy to use in practice, since one must first compute all eigenvalues and then verify a collection of required \(2\)-adic conditions, while the transfer time is specified only implicitly through the graph-dependent parameter \(M\) defined in Lemma \ref{lem:iff of PST} . 	
By contrast, our results provide a direct and explicit structural characterization of the connection set, and determine the perfect state transfer time explicitly. 
In particular, for a fixed dihedral group, constructing Cayley graphs with perfect state transfer from the eigenvalue criterion alone is a highly nontrivial task, whereas our result turns this into an immediate and systematic procedure.
\end{rem}

\subsection{The case \(n\equiv 2\pmod 4\)}

Write \(n=2m\), where \(m\) is odd. 
Then \(D=D_0\cup D_1\). 
The first step is to determine the odd-indexed eigenvalues, as in Lemma \ref{thm:-1 0} for the case \(n\equiv0\pmod4\).

\begin{lem}\label{lem:n2-odd-eta}
	Let $n \equiv 2 \pmod{4}$, $G \cong D_{2n}$ and $\Gamma = \operatorname{Cay}(G,S)$. 
	If $\Gamma$ admits perfect state transfer, then for every odd $h$ with $1 \le h \le \frac{n}{2} - 1$, $\eta_h = -1$ if $\frac{n}{2} \in D_1$, and $\eta_h = 0$ if $\frac{n}{2} \notin D_1$. 
	Consequently, $D_0 = 2(D_1 \setminus \{\frac{n}{2}\})$.
\end{lem}
\begin{proof}
Put $m = \frac{n}{2}$. Since $n \equiv 2 \pmod{4}$, we have $m$ odd and $D = D_0 \cup D_1$. 
Let $X := \operatorname{ICG}_n(D)$, and denote its eigenvalues by $\theta_j = \sum\limits_{d \in D} c(j, \frac{n}{d})$, $0 \le j \le n-1$.
For every odd $h$ with $1 \le h \le m-1$, we have $\theta_h = \eta_h$.
Since $X$ is undirected, $\theta_{n-h} = \theta_h = \eta_h$. 
Thus the only odd-position eigenvalue of $X$ not represented by some such $\eta_h$ is $\theta_m$.

We first record the parity of $\theta_m$. 
Since $\theta_m = \sum\limits_{a^k \in S_0} (-1)^k$, we have $\theta_m \equiv |S_0| \pmod{2}$. 
Moreover, $|S_0| = \sum\limits_{d \in D} \varphi(\frac{n}{d})$.
Here $\varphi(\frac{n}{d})$ is odd only when $\frac{n}{d} = 2$, that is, only when $d = m$. 
Hence
\[
\theta_m \equiv |S_0| \equiv
\begin{cases}
	1 \pmod{2}, & m \in D_1, \\
	0 \pmod{2}, & m \notin D_1.
\end{cases}
\]

By Corollary \ref{core:parity}, all $\eta_h$ with $h$ odd have the same parity. 
We now determine the parity of $\theta_m$ according to whether $m$ belongs to $D_1$ or not.

If $m \in D_1$, then $|S_0|$ and $\theta_m$ are odd. We claim that all $\eta_h$ with $h$ odd are odd. 
Suppose, to the contrary, that they are all even. 
In particular, $\eta_1$ is even. 
Write $A := |S_0|$, $T := \sum\limits_{a^k \in S_0} (-1)^k$.
	Then \(A\) and \(T\) are both odd. 
	By Corollary \ref{core:eigenvalues of D2n},
	\[
	\lambda_1=A+m(\alpha+\beta),~~
	\lambda_2=A-m(\alpha+\beta),~~
	\lambda_3=T+m(\alpha-\beta),~~
	\lambda_4=T+m(-\alpha+\beta).
	\]
	If \(\alpha+\beta\) is even, then \(\lambda_1\) is odd. Since \(\eta_1\) is even, we get $v_2(\lambda_1-\eta_1)=0$.
	However, $\lambda_1-\lambda_3=A-T+2\beta m$	is even, contradicting Lemma \ref{lem:iff of PST}, which requires $	v_2(\lambda_1-\eta_1)=v_2(\lambda_1-\lambda_3)$.
	If \(\alpha+\beta\) is odd, then \(\alpha+\beta=1\), and \(\lambda_1\) is even. 
	Hence $v_2(\lambda_1-\eta_1)\ge 1$.
	But $\lambda_1-\lambda_2=2m$, so $v_2(\lambda_1-\lambda_2)=1$, which contradicts the condition $v_2(\lambda_1-\lambda_2)>v_2(\lambda_1-\eta_1)$	in Lemma \ref{lem:iff of PST}. 
	Thus all \(\eta_h\) with \(h\) odd are odd.
	Since \(\theta_m\) is also odd, all odd-position eigenvalues of \(X=\mathrm{ICG}_n(D)\)
	are odd. 
	By Lemma \ref{lem:basic lem12+13}-(2), we obtain $D=D_1^*\cup 2D_1^*\cup\{m\}$, $D_1^*:=D_1\setminus\{m\}$.
	Therefore, for odd \(h\),
	\[
	\eta_h
	=
	\sum_{d\in D_1^*}
	\left(
	c\left(h,\frac nd\right)
	+
	c\left(h,\frac n{2d}\right)
	\right)
	+
	c(h,2).
	\]
	For \(d\in D_1^*\),  \(\frac{n}{2d}\) is odd.
	Since \(h\) is odd, Lemma \ref{lem:ramanujan-odd}-(1) gives $c\left(h,\frac nd\right)=-c\left(h,\frac n{2d}\right)$.
	Thus all paired terms cancel. 
	Moreover, \(c(h,2)=-1\). 
	Hence $\eta_h=-1$ for every odd \(h\).
	
	If \(m\notin D_1\), then \(|S_0|\) and \(\theta_m\) are even. We claim that all
	\(\eta_h\) with \(h\) odd are even. Suppose, to the contrary, that they are all odd. In particular,
	\(\eta_1\) is odd. With the notation above, \(A\) and \(T\) are both even.
	If \(\alpha+\beta\) is even, then \(\lambda_1\) is even. Since \(\eta_1\) is odd, we get $v_2(\lambda_1-\eta_1)=0$.
	But $\lambda_1-\lambda_3=A-T+2\beta m$ is even, again contradicting Lemma \ref{lem:iff of PST}. 
	If \(\alpha+\beta\) is odd, then \(\alpha+\beta=1\), and \(\lambda_1\) is odd. 
	Since \(\eta_1\) is odd, we have $v_2(\lambda_1-\eta_1)\ge 1$.
	On the other hand, $\lambda_1-\lambda_2=2m$, so \(v_2(\lambda_1-\lambda_2)=1\), contradicting $v_2(\lambda_1-\lambda_2)>v_2(\lambda_1-\eta_1)$.
	Therefore all \(\eta_h\) with \(h\) odd are even.
	Since \(\theta_m\) is also even, all odd-position eigenvalues of \(X=\mathrm{ICG}_n(D)\)
	are even. 
	By Lemma \ref{lem:basic lem12+13}-(1), $D=D_1\cup 2D_1$.
	Hence, for odd \(h\),
	\[
	\eta_h
	=
	\sum_{d\in D_1}
	\left(
	c\left(h,\frac nd\right)
	+
	c\left(h,\frac n{2d}\right)
	\right).
	\]
	Again \(\frac{n}{2d}\) is odd and $c(h,\frac nd)=-c(h,\frac n{2d})$.
	Thus  $\eta_h=0$ for every odd \(h\).
	
	Finally, in the first case we have $D=D_1^*\cup 2D_1^*\cup\{m\}$, and in the second case we have $D=D_1\cup 2D_1$.
	Since \(D=D_0\cup D_1\), both cases give $D_0=2(D_1\setminus\{\frac n2\})$.
	This completes the proof.
\end{proof}

We next show that the case \(\frac n2\in D_1\) in Lemma \ref{lem:n2-odd-eta} is impossible.

\begin{pro}\label{prop:n2-no-special}
	Let $n\equiv 2 \pmod 4$, $G\cong D_{2n}$, and $\Gamma=\mathrm{Cay}(G,S)$ is connected. 
	If $\Gamma$ admits perfect state transfer, then $\frac{n}{2} \notin D_1$ and $D_0 = 2D_1$.
\end{pro}
\begin{proof}
	Put \(m=\frac n2\). 
	Assume, to the contrary, that \(m\in D_1\). 
	Then, by Lemma \ref{lem:n2-odd-eta}, $D_0=2(D_1\setminus\{m\})$ and $\eta_1=-1$.
	
	Write $A:=|S_0|$, $T:=\sum\limits_{a^k\in S_0}(-1)^k$.
	For each \(d\in D_1\setminus\{m\}\), the two divisors \(d\) and \(2d\)
	occur together. 
	Since \(\frac n{2d}>1\) is an odd integer, their total contribution to \(A\) is $\varphi(\frac nd)+\varphi(\frac n{2d})=2\varphi(\frac n{2d})\in 4\mathbb Z$,
	Their total contribution to \(T\) is \(0\). 
	The divisor \(m\) contributes \(1\) to \(A\) and \(-1\) to \(T\). 
	Hence $A\equiv 1\pmod 4$, $T=-1$.
	By Corollary \ref{core:eigenvalues of D2n}, $\lambda_1=A+m(\alpha+\beta)$, $\lambda_3=-1+m(\alpha-\beta)$, $\lambda_4=-1+m(-\alpha+\beta)$.
	Thus $\lambda_1-\lambda_3=A+1+2m\beta$, $\lambda_1-\lambda_4=A+1+2m\alpha$.
	Moreover, $\lambda_1-\eta_1=\lambda_1+1=A+1+m(\alpha+\beta)$.
	
	If \(\alpha+\beta\) is odd, then \(\lambda_1-\eta_1\) is odd, while 	\(\lambda_1-\lambda_3\) and \(\lambda_1-\lambda_4\) are even. 
	This contradicts Lemma \ref{lem:iff of PST}. 
	Hence \(\alpha+\beta\) is even, so \(\alpha=\beta\).
	If \(\alpha=\beta=1\), then $\lambda_1-\eta_1=A+1+2m\equiv 0\pmod 4$, and hence $v_2(\lambda_1-\eta_1)\ge 2$.
	However, $\lambda_1-\lambda_2=4m$, so \(v_2(\lambda_1-\lambda_2)=2\), contradicting the condition $v_2(\lambda_1-\lambda_2)>v_2(\lambda_1-\eta_1)$ in Lemma \ref{lem:iff of PST}. Therefore \(\alpha=\beta=0\).
	Thus \(S_1=\emptyset\), contradicts the assumption that \(\Gamma\) is connected.
	Therefore \(m\notin D_1\). By Lemma \ref{lem:n2-odd-eta} , we obtain $D_0=2D_1$.
	This completes the proof.
\end{proof}

By Proposition~\ref{prop:n2-no-special}, we have \(\frac n2\notin D_1\) and \(D_0=2D_1\). 
We now use these two conditions to describe \(S_0\) through an auxiliary integral circulant graph of order \(\frac n2\).

\begin{lem}\label{lem:n2-reduction}
	Let \(n=2m\), where \(m\) is odd.
	Assume that \(\Gamma\) is connected and admits perfect state transfer. 
	Then \(\Gamma\) is integral, and hence $S_0=\bigcup\limits_{d\in D}G_n(d)$, where \(D\) is a set of proper divisors of \(n\). 
	Set \(E:=D_1\), and let $\tau_j=\sum\limits_{e\in E}c(j,\frac{m}{e})$, $0\le j\le m-1$, be the eigenvalues of \(\mathrm{ICG}_m(E)\). 
	Then, for \(1\le j\le \frac{m-1}{2}\), $\eta_{2j-1}=0$, $\eta_{2j}=2\tau_j$.
	Moreover, $|S_0|=2\tau_0$, $\sum\limits_{a^k\in S_0}(-1)^k=0$.
	In particular, \(\tau_0\) is even.
\end{lem}
\begin{proof}
	Since \(n\equiv 2\pmod 4\), we have \(D=D_0\cup D_1\). 
	By Proposition~\ref{prop:n2-no-special}, \(m=\frac n2\notin D_1\) and \(D_0=2D_1\). 
	Hence \(E=D_1\subseteq\{e:e\mid m,\ e<m\}\).
	It follows that \(D=E\cup 2E\), and hence $S_0=\bigcup\limits_{e\in E}(G_n(e)\cup G_n(2e))$.
	
	Let \(h\) be odd. Then
	\[
	\eta_h=\sum_{e\in E}\left(c \left(h,\frac{n}{e}\right)+c \left(h,\frac{n}{2e}\right)\right)
	=\sum_{e\in E}\left(c \left(h,\frac{2m}{e}\right)+c \left(h,\frac{m}{e}\right)\right).
	\]
	Since \(\frac{m}{e}\) is odd, Lemma \ref{lem:ramanujan-odd}-(1) gives $c \left(h,\frac{2m}{e}\right)=-c \left(h,\frac{m}{e}\right)$.
	Therefore \(\eta_h=0\) for every odd \(h\), and in particular $\eta_{2j-1}=0$, $1\le j\le \frac{m-1}{2}$.
	
	Now let \(1\le j\le \frac{m-1}{2}\). We obtain $\eta_{2j}=\sum\limits_{e\in E}(c(2j,\frac{2m}{e})+c(2j,\frac{m}{e}))$.
	By Lemma \ref{lem:basic lem 17}, $c(2j,\frac{2m}{e})=c (j,\frac{m}{e})$, and Lemma \ref{lem:ramanujan-odd}-(3) gives $c(2j,\frac{m}{e})=c(j,\frac{m}{e})$.
	Hence $\eta_{2j}=2\sum\limits_{e\in E}c \left(j,\frac{m}{e}\right)=2\tau_j$.
	
	Next,
	\[
	|S_0|
	=\sum_{e\in E}|G_n(e)|+\sum_{e\in E}|G_n(2e)|
	=\sum_{e\in E}\varphi \left(\frac{2m}{e}\right)+\sum_{e\in E}\varphi \left(\frac{m}{e}\right)
	=2\sum_{e\in E}\varphi \left(\frac{m}{e}\right)
	=2\tau_0.
	\]
	For each \(e\in E\), the integer \(e\) is odd. 
Hence \(a^k\in G_n(e)\) implies that
\(k\) is odd, while \(a^k\in G_n(2e)\) implies that \(k\) is even. 
	Moreover, $|G_n(e)|=\varphi(\frac{n}{e})=\varphi (\frac{2m}{e})=\varphi(\frac{m}{e})=\varphi(\frac{n}{2e})=|G_n(2e)|$.
	Therefore, for every \(e\in E\),
	\[
	\sum_{a^k\in G_n(e)}(-1)^k+\sum_{a^k\in G_n(2e)}(-1)^k=0.
	\]
	Summing over all \(e\in E\), we get $\sum\limits_{a^k\in S_0}(-1)^k=0$.
	
	Finally, for \(e\in E\), \(\frac{m}{e}>1\) is odd, and hence
	\(\varphi \left(\frac{m}{e}\right)\) is even. Therefore $\tau_0=\sum\limits_{e\in E}\varphi \left(\frac{m}{e}\right)$ is even. 
	This completes the proof.
\end{proof}

We next determine the possible structure of \(S_1\).

\begin{pro}\label{prop:n2-S1}
	Let $n=2m$ with $m$ odd, $G\cong D_{2n}$, $\Gamma=\mathrm{Cay}(G,S)$ be connected. 
	If $\Gamma$ admits
	perfect state transfer, then $S_1=b\langle a\rangle$.
\end{pro}
\begin{proof}
	Assume that $\Gamma$ admits perfect state transfer. By Proposition \ref{prop:n2-no-special} and
	Lemma~\ref{lem:n2-reduction}, there exists a set $E\subseteq \{e:e\mid m,\ e<m\}$ such that $D=E\cup 2E$, $|S_0|=2\tau_0$, $\sum\limits_{a^k\in S_0}(-1)^k=0$, $\eta_{2j-1}=0$, $\eta_{2j}=2\tau_j$, and \(\tau_0\) is even.
	
	By Corollary \ref{core:eigenvalues of D2n}, $\lambda_1=2\tau_0+m(\alpha+\beta)$, $\lambda_2=2\tau_0-m(\alpha+\beta)$, $\lambda_3=m(\alpha-\beta)$, $\lambda_4=m(-\alpha+\beta)$.	
	We first show that $\alpha=\beta$. 
	Suppose otherwise. 
	Then exactly one of \(\alpha,\beta\) is equal to \(1\). Since \(m\) is odd and \(\tau_0\) is even, $\lambda_1-\lambda_3=2\tau_0+2\beta m$, $\lambda_1-\lambda_4=2\tau_0+2\alpha m$.
	Hence one of these two integers is congruent to \(2\pmod 4\), while the other is divisible by \(4\).
	Therefore they have different \(2\)-adic valuations, contradicting Lemma \ref{lem:iff of PST}-(1), which requires $v_2(\lambda_1-\lambda_3)=v_2(\lambda_1-\lambda_4)$.
	Thus \(\alpha=\beta\). 
	If \(\alpha=\beta=0\), then \(S_1=\emptyset\).
	Hence \(\Gamma\) is disconnected, a contradiction. 
	Therefore \(\alpha=\beta=1\), and so \(S_1=b\langle a\rangle\).
	This completes the proof.
\end{proof}

Combining Proposition \ref{prop:n2-S1} with Lemma \ref{lem:auxiliary-Km}, we now obtain a complete structural characterization of perfect state transfer in the case \(n\equiv 2 \pmod 4\).

\begin{thm}\label{thm:n=2mod4}
	Let \(n\equiv 2 \pmod 4\), $G\cong D_{2n}$, $G=\langle a,b: a^n=b^2=1,\ b^{-1}ab=a^{-1}\rangle$ and \(S\subseteq G\setminus\{1_G\}\) be conjugation-closed. 
	Assume that \(\Gamma=\mathrm{Cay}(G,S)\) is connected. 
	Then \(\Gamma\) admits perfect state transfer if and only if $S=G\setminus\{1_G,a^\frac{n}{2}\}$.
\end{thm}
\begin{proof}
	Let \(m=\frac{n}{2}\).
	Clearly $m$ is odd.

	\medskip
\noindent\emph{Necessity.}
	Assume that \(\Gamma\) admits perfect state transfer.
	Then Lemma \ref{lem:iff of PST} gives that \(\Gamma\) is integral and Proposition \ref{prop:n2-S1} gives $S_1=b\langle a\rangle$. 
	By Corollary \ref{core:eigenvalues of D2n},	there exists a set \(D\subseteq \{d:d\mid n,\ d<n\}\) such that $S_0=\bigcup\limits_{d\in D}G_n(d)$.	
	By Proposition \ref{prop:n2-no-special}, \(\frac{n}{2}\notin D_1\) and \(D_0=2D_1\).
	Set \(E:=D_1\). 
	Since \(m\) is odd, it follows that $E\subseteq \{e:e\mid m,\ e<m\}$, $D=E\cup 2E$. 
	Therefore
	\[
	S_0=\bigcup_{e\in E}\bigl(G_n(e)\cup G_n(2e)\bigr).
	\]
	Let \(\tau_j\) be the eigenvalues of \(\mathrm{ICG}_m(E)\), that is $\tau_j=\sum\limits_{e\in E}c(j,\frac{m}{e})$, $0\le j\le m-1$.
	By Lemma \ref{lem:n2-reduction} together with Corollary \ref{core:eigenvalues of D2n}, we have $\lambda_1=2\tau_0+2m$, $\lambda_2=2\tau_0-2m$, $\lambda_3=\lambda_4=0$, and \(\tau_0\) is even.
	
	Since \(m\) is odd and \(\tau_0\) is even, \(\tau_0+m\) is odd. 
	Hence $v_2(\lambda_1-\eta_1)=v_2(\lambda_1)=v_2\left(2(\tau_0+m)\right)=1$.
	Thus Lemma~\ref{lem:iff of PST}-(1) yields \(\gamma=1\). 
	For \(1\le j\le \frac{m-1}{2}\), Lemma~\ref{lem:n2-reduction} gives $\eta_{2j-1}=0$, $\eta_{2j}=2\tau_j$.
	Therefore $	1=\gamma=v_2(\eta_{2j}-\eta_{2j-1})=v_2(2\tau_j)$.
	Hence each \(\tau_j\) is odd for \(1\le j\le \frac{m-1}{2}\). 
	By the symmetry \(\tau_j=\tau_{m-j}\), every \(\tau_j\) with \(1\le j\le m-1\) is odd.
	Therefore, by Lemma~\ref{lem:auxiliary-Km}, $E=\{e:e\mid m,\ e<m\}$.
	Hence
	\[
	S_0=\bigcup_{\substack{e\mid m,~e<m}}\bigl(G_n(e)\cup G_n(2e)\bigr).
	\]
	
	We now show that $S_0=\langle a\rangle\setminus\{1,a^m\}$. 
	Indeed, let \(1\le k< n\). 
	If \(a^k\neq 1,a^m\), then \(\gcd(k,n)\neq n,m\). 
	Since \(n=2m\) with \(m\) odd, we may write \(\gcd(k,n)=e\) or \(2e\) for some divisor \(e\) of \(m\) with \(e<m\).
	Thus \(a^k\in G_n(e)\cup G_n(2e)\) for some \(e\mid m\), \(e<m\). 	
	Conversely, if \(a^k\in G_n(e)\cup G_n(2e)\) for some \(e\mid m\), \(e<m\), then \(\gcd(k,n)\neq n,m\), and hence \(a^k\neq 1,a^m\). 
	Therefore $S_0=\langle a\rangle\setminus\{1,a^m\}$.
	Together with \(S_1=b\langle a\rangle\), this proves the necessity.
	
	\medskip
\noindent\emph{Sufficiency.}
	Assume that $S_0=\langle a\rangle\setminus\{1_G,a^m\}$ and $S_1=b\langle a\rangle$.
	Set $E:=\{e:e\mid m,\ e<m\}$.
	Then $S_0=\bigcup\limits_{e\in E}\left(G_n(e)\cup G_n(2e)\right)$.
	Moreover, $\bigcup\limits_{e\in E}G_m(e)=\mathbb Z_m\setminus\{0\}$, and hence \(\mathrm{ICG}_m(E)=K_m\).  
	Therefore its eigenvalues are $\tau_0=m-1$, $\tau_j=-1$, $1\le j\le m-1$. 
	
Since \(S_0=\bigcup\limits_{e\in E}(G_n(e)\cup G_n(2e))\), for odd \(h\) we have
\[
\eta_h=\sum_{e\in E}\left(c\left(h,\frac{2m}{e}\right)+c\left(h,\frac{m}{e}\right)\right)=0.
\]
Thus $\eta_{2j-1}=0$, $1\le j\le \frac{m-1}{2}$.
For \(1\le j\le \frac{m-1}{2}\), Lemmas~\ref{lem:ramanujan-odd} and~\ref{lem:basic lem 17} give
\[
\eta_{2j}
=\sum_{e\in E}\left(c\left(2j,\frac{2m}{e}\right)+c\left(2j,\frac{m}{e}\right)\right)
=2\sum_{e\in E}c\left(j,\frac{m}{e}\right)
=2\tau_j.
\]
Together with $\tau_j=-1$, $1\le j\le m-1$, we have $\eta_{2j}=-2$, $1\le j\le \frac{m-1}{2}$.
Moreover, since \(S=G\setminus\{1_G,a^m\}\), we have $|S_0|=2m-2$, $\sum\limits_{a^k\in S_0}(-1)^k=0$.

By Corollary~\ref{core:eigenvalues of D2n}, we get $\lambda_1=4m-2$, $\lambda_2=-2$, $\lambda_3=\lambda_4=0$.
Since \(m\) is odd, $v_2(\lambda_1-\lambda_3)=v_2(\lambda_1-\lambda_4)=v_2(\lambda_1-\eta_1)=1$.
Moreover, $v_2(\lambda_1-\lambda_2)=v_2(4m)>1$.
Finally, the sequence $\eta_1,\eta_2,\ldots,\eta_{m-1}$ alternates between \(0\) and \(-2\). 
Hence $v_2(\eta_{h+1}-\eta_h)=1$, $1\le h\le m-2$.
Therefore all conditions in Lemma~\ref{lem:iff of PST}-(1) are satisfied with \(\gamma=1\), and so \(\Gamma\) admits perfect state transfer.
\end{proof}

\begin{core}\label{core: minimal time 2}
	Let \(n\equiv 2 \pmod 4\), $G\cong D_{2n}$, and $\Gamma=\mathrm{Cay}(G,S)$  be connected.
	If \(\Gamma\) admits perfect state transfer, then  the minimal perfect state transfer time is \(\frac{\pi}{2}\).
\end{core}
\begin{proof}
	Assume that \(\Gamma\) admits perfect state transfer.
	By Theorem \ref{thm:n=2mod4} and its proof, we have $\gamma=1$ and $\lambda_1=4m-2$, $\lambda_2=-2$, $\lambda_3=\lambda_4=0$, $\eta_{2j-1}=0$, $\eta_{2j}=-2$.
	Hence $M=\gcd(\lambda_1-\lambda:\lambda\neq\lambda_1)=2$, and Lemma \ref{lem:iff of PST} gives the minimum perfect state transfer time is $\frac{\pi}{2}$.
\end{proof}

\begin{rem}\label{rem:2}
Compared with Lemma~\ref{lem:iff of PST}, Theorem \ref{thm:n=2mod4} gives a stronger conclusion in the case \(n\equiv 2\pmod 4\), as it identifies the unique admissible conjugation-closed connection set and hence reduces the existence of perfect state transfer to an immediate structural check.	
\end{rem}

\section{Examples}\label{sec:examples}

We finish the paper with examples illustrating how our results provide convenient and concise criteria for deciding the existence of perfect state transfer, and at the same time enable explicit constructions of Cayley graphs over dihedral groups with perfect state transfer.

\begin{example}\label{ex:positive-central-involution}
	Assume that \(n\equiv 0\pmod 4\), \(G\cong D_{2n}\), $G=\langle a,b: a^n=b^2=1,\ b^{-1}ab=a^{-1}\rangle$, and \(S=S_0\cup S_1\), where $S_0=\{a^{\frac n2}\}$, $S_1=b\langle a\rangle$.
	Then \(\Gamma=\mathrm{Cay}(G,S)\) admits perfect state transfer.
\end{example}

\begin{proof}
	If \(8\mid n\), then Theorem~\ref{thm:8midn-characterization} applies with $E=\emptyset$, $F=\emptyset$, $\xi=\frac n2$.
	If \(n\equiv 4\pmod 8\), then Theorem~\ref{thm:n4mod8-full-reflection} applies with $E=\emptyset$, $\xi=\frac n2$, and \(S_1=b\langle a\rangle\).
	In both cases, $S_0=\bigcup\limits_{d\in D}G_n(d)$, $D=\{\frac n2\}$.
	Therefore \(\Gamma\) admits perfect state transfer. 
\end{proof}

\begin{example}
Let \(p\) be an odd prime, let \(n=2p\), and let $G\cong D_{2n}$, $G=\langle a,b: a^n=b^2=1,\ b^{-1}ab=a^{-1}\rangle$.
It follows immediately from Theorem \ref{thm:n=2mod4} that the unique Cayley graph over \(G\) with conjugation-closed connection set admitting perfect state transfer is $\mathrm{Cay}(G,(\langle a\rangle\setminus\{1,a^p\})\cup b\langle a\rangle)$.
\end{example}

\section{Concluding remarks}\label{sec:7}

In this paper, we have characterized the conjugation-closed connection sets for which
connected Cayley graphs over dihedral groups admit perfect state transfer. 
As a consequence, we also determined that the minimum perfect state transfer time for this class of graphs is \(\frac{\pi}{2}\).
Our approach combines character-theoretic eigenvalue formulas, Ramanujan sums,
M\"obius inversion, and arguments based on the \(p\)-adic exponential valuation.
A natural next step is to study the non-conjugation-closed case.
\begin{que}
	Characterize the connection sets of Cayley graphs over dihedral groups with non-conjugation-closed connection sets that admit perfect state transfer.
\end{que}

More broadly, these methods may also be useful for studying connection sets admitting perfect state transfer in Cayley graphs over other classes of groups.

\section*{Data availability statement}
No datasets were generated or analysed during the current study.

\end{document}